\documentclass[reqno,11pt]{amsart}

\usepackage{amsthm} %with nospthms of svjour3
\usepackage{amssymb}
\usepackage{amsfonts}
\usepackage{mathrsfs}
\usepackage{bm}
\usepackage{yhmath}
\usepackage{tikz}
\DeclareMathAlphabet{\mathcal}{OMS}{cmsy}{m}{n}

\theoremstyle{plain}
\newtheorem{theorem}{Theorem}[section]
\newtheorem{lemma}[theorem]{Lemma}
\newtheorem{proposition}[theorem]{Proposition}
\newtheorem{corollary}[theorem]{Corollary}

\newtheorem*{BSEtheorem}{Theorem \ref{thm:BSE theorem}}
\newtheorem*{maintheorem}{Theorem \ref{thm:main theorem}}
\newtheorem*{maintheoremII}{Theorem \ref{thm:main theoremII}}

\newtheorem{claim}{Claim}

\theoremstyle{definition}
\newtheorem{definition}[theorem]{Definition}

\theoremstyle{remark}
\newtheorem{remark}[theorem]{Remark}

\numberwithin{equation}{section}

\allowdisplaybreaks

\DeclareMathAlphabet{\mathcal}{OMS}{cmsy}{m}{n}
\makeatletter
\def\@citestyle{\m@th\upshape\mdseries}
\def\citeform#1{{\bfseries#1}}
\def\@cite#1#2{{%
  \@citestyle[\citeform{#1}\if@tempswa, #2\fi]}}
\@ifundefined{cite }{%
  \expandafter\let\csname cite \endcsname\cite
  \edef\cite{\@nx\protect\@xp\@nx\csname cite \endcsname}%
}{}
\makeatother

\renewcommand{\leq}{\leqslant}
\renewcommand{\geq}{\geqslant}

\newcommand{\inner}[2]{\langle #1\,,#2\rangle}
\newcommand{\biginner}[2]{\left\langle #1\,,\,#2\right\rangle}

\newcommand{\ddl}[2]{\frac{d{#1}}{d{#2}}}

\newcommand{\ppl}[2]{\frac{\partial{#1}}{\partial{#2}}}
\newcommand{\ppz}[2]{\frac{\partial^2{#1}}{\partial{#2}^2}}

\renewcommand{\Im}{\mathop{\mathrm{Im}}}
\renewcommand{\Re}{\mathop{\mathrm{Re}}}

\newcommand{\C}{\mathbb{C}}
\newcommand{\R}{\mathbb{R}}

\newcommand{\B}{\mathbb{B}}
\renewcommand{\H}{\mathbb{H}}
\newcommand{\Lbb}{\mathbb{L}}
\newcommand{\U}{\mathbb{U}}

\newcommand{\HH}{\mathbf{H}}
\newcommand{\T}{\mathbf{T}}

\DeclareMathOperator{\dist}{dist}%
\DeclareMathOperator{\I}{I}%
\DeclareMathOperator{\II}{II}%

\DeclareMathOperator{\Area}{Area}%
\DeclareMathOperator{\Length}{Length}%
\DeclareMathOperator{\Ric}{Ric}%
\DeclareMathOperator{\Span}{span}%

\DeclareMathOperator{\sech}{sech}%

\newcommand{\CC}{\mathcal{C}}
\newcommand{\Hcal}{\mathcal{H}}
\newcommand{\Lcal}{\mathcal{L}}

\newcommand{\Cscr}{\mathscr{C}}
\newcommand{\Nscr}{\mathscr{N}}

\newcommand{\Mob}{\mathsf{M\ddot{o}b}}%
\renewcommand{\O}{\mathsf{O}}%
\newcommand{\PSL}{\mathsf{PSL}}%
\newcommand{\SO}{\mathsf{SO}}%

\newcommand{\abar}{\bar{a}}
\newcommand{\atilde}{\tilde{a}}

\makeatletter
\newcommand*\rel@kern[1]{\kern#1\dimexpr\macc@kerna}
\newcommand*\widebar[1]{%
  \begingroup
  \def\mathaccent##1##2{%
    \rel@kern{0.8}%
    \overline{\rel@kern{-0.8}\macc@nucleus\rel@kern{0.2}}%
    \rel@kern{-0.2}%
  }%
  \macc@depth\@ne
  \let\math@bgroup\@empty \let\math@egroup\macc@set@skewchar
  \mathsurround\z@ \frozen@everymath{\mathgroup\macc@group\relax}%
  \macc@set@skewchar\relax
  \let\mathaccentV\macc@nested@a
  \macc@nested@a\relax111{#1}%
  \endgroup
}
\makeatother
\begin{document}

\title[Stability of Catenoids and Helicoids]
{Stability of Catenoids and Helicoids in Hyperbolic Space}
%\subtitle{Least area Spherical Catenoids in $\H^3$}

\author{Biao Wang}
\thanks{This research was partially supported by
PSC-CUNY Research Award \#{}68119-0046.}

%\date{}
\date{\today}

\subjclass{53A10}
%\subjclass{Primary 53A10, Secondary 53C42}
% 57M05 Low-dimensional topology
% 53A10 Minimal surfaces, surfaces with prescribed mean curvature
% 53C42 Immersions (minimal, prescribed curvature, tight, etc.)
% 53C44 Geometric evolution equations (mean curvature flow)
\address{Department of Mathematics and Computer Science\\
           Queensborough Community College, The City University of New York\\
           222-05 56th Avenue Bayside, NY 11364\\}
\email{biwang@qcc.cuny.edu}

\begin{abstract}
  In this paper, we study the stability of catenoids and helicoids in
  the hyperbolic $3$-space $\H^3$. We will prove the following results.
  \begin{enumerate}
    \item For a family of spherical minimal catenoids $\{\CC_a\}_{a>0}$
          in the hyperbolic $3$-space $\H^3$ (see
          $\S$\ref{subsec:catnoids} for detail definitions),
          there exist two constants $0<a_c<a_l$ such that
          \begin{itemize}
            \item $\CC_a$ is an unstable minimal surface
                  with Morse index one if $a<a_c$,
            \item $\CC_a$ is a globally stable minimal surface
                  if $a\geq{}a_c$, and
            \item $\CC_a$ is a least area minimal surface
                  in the sense of Meeks and Yau (see
                  $\S$\ref{subsec:minimal surfaces} for the
                  definition) if $a\geq{}a_l$.
          \end{itemize}
    \item For a family of minimal helicoids $\{\Hcal_{\abar}\}_{\abar\geq{}0}$
          in the hyperbolic space $\H^3$ (see $\S$\ref{sec:helicoid-II} for
          detail definitions), there exists a constant
          $\abar_{c}=\coth(a_c)$ such that
          \begin{itemize}
             \item $\Hcal_{\abar}$ is a globally stable minimal surface
                   if $0\leq\abar\leq\abar_c$, and
             \item $\Hcal_{\abar}$ is an unstable minimal surface
                   with Morse index infinity if $\abar>\abar_c$.

  \end{itemize}
  \end{enumerate}
\end{abstract}

\maketitle

%\tableofcontents%

%====================================================================
\section{Introduction}\label{sec:Introduction}

The study of the catenoid and the helicoid in the $3$-dimensional Euclidean
space $\R^3$ can be traced back to Leonhard Euler and Jean Baptiste Meusnier
in the 18th century. Since then mathematicians have found many properties of
the catenoid and the helicoid in $\R^3$.
The first property is that both the catenoid and the helicoid in $\R^3$
are unstable. Actually do Carmo and Peng \cite{dCP79} proved that the plane
is the \emph{unique} stable complete minimal surface in $\R^3$.
Let's list some other properties here (all minimal surfaces are in $\R^3$):
\begin{enumerate}
  \item The plane and the catenoid are the only minimal surface
        of revolution in $\R^3$
        (Bonnet in 1860, see \cite[$\S$2.5]{MP2012}).
  \item The catenoid is the unique embedded complete  minimal
        surface in $\R^3$ with finite topology and with two ends
        \cite{Sch83}.
  \item The catenoid and the Enneper's surface are the only
        orientable complete minimal surfaces in $\R^3$ with
        Morse index equal to one \cite{LR89}.
  \item The plane and the catenoid are the only embedded
        complete minimal surfaces of finite total curvature
        ($=-4\pi$) and genus zero in $\R^3$ \cite{LR91}.
  \item The plane and the helicoid are the only ruled minimal
        surfaces in $\R^3$ (Catalan in 1842,
        see \cite[$\S$2.5]{MP2012} or \cite[pp.34--35]{FT91}).
  \item The helicoid in $\R^3$ has genus zero, one end, infinite total
        curvature \cite[$\S$2.5]{MP2012} and infinite Morse
        index \cite[p.199]{Tuz93}.
  \item The plane and the helicoid are the unique
        simply-connected, complete, embedded
        minimal surface in $\R^3$ \cite{MR2005}.
\end{enumerate}
The reader can see the survey \cite{MP2011} and the book
\cite{MP2012} for more properties of the catenoid and the helicoid,
and the references cited therein.

In this paper we will study the stability of the \emph{spherical}
catenoids and the helicoids in the hyperbolic $3$-space.
There are three models of the hyperbolic $n$-space (see
$\S$\ref{subsec:hyperbolic space models} for definitions),
but we may use the notation $\H^n$ to denote the hyperbolic
$n$-space without emphasizing the model.
We list some properties of the catenoids and helicoids in $\H^3$:
\begin{enumerate}
  \item Each catenoid (hyperbolic, parabolic or spherical) is a
        complete embedded
        minimal surface in $\H^3$ (see \cite[Theorem (3.26)]{dCD83}).
  \item Each spherical catenoid in $\H^3$ has finite total curvature
        (see the computation in \cite[p.708]{dCD83}).
  \item All hyperbolic and parabolic catenoids in $\H^3$ are least
        area minimal surfaces (see \cite[p. 3574]{Can07}), so all of
        them are globally stable.
  \item Each helicoid in $\H^3$ is a complete embedded ruled
        minimal surface
        (see \cite[Theorem 1]{Mor82} and \cite[pp.221--222]{Tuz93}).
  \item Each helicoid in $\H^3$ has infinite total curvature
        (see the computation in \cite[p.60]{Mor82}).
  \item The stability of the spherical catenoids and helicoids in
        $\H^3$ is characterized by Theorem~\ref{thm:BSE theorem}
        and Theorem~\ref{thm:main theoremII} respectively.
\end{enumerate}
The reader can see the paper \cite{Tuz93} for more properties of
the catenoids and the helicoids in the hyperbolic $3$-space.

It was Mori \cite{Mor81} who studied the spherical catenoids in the
hyperboloid model of the hyperbolic space $\H^3$ at first.
Then Do Carmo and Dajczer \cite{dCD83} studied three types of
rotationally symmetric minimal hypersurfaces in the hyperboloid
model of the hyperbolic space $\H^{n+1}$ (see
also $\S$\ref{subsec:catenoids in hyperboloid model}
for a brief description).
A rotationally symmetric minimal hypersurface is called a
{\em spherical} catenoid if it is foliated by spheres, a
{\em hyperbolic} catenoid if it is foliated by totally geodesic
hyperplanes, or a {\em parabolic} catenoid if it is foliated
by horospheres. Do Carmo and Dajczer
proved that the hyperbolic and parabolic catenoids in $\H^3$ are
globally stable (see \cite[Theorem 5.5]{dCD83}), then Candel proved
that the hyperbolic and parabolic catenoids in $\H^3$ are least area
minimal surfaces (see \cite[p. 3574]{Can07}).

Compared with the hyperbolic and parabolic catenoids, the spherical
catenoids in $\H^3$ are more complicated.
Let $\CC_a$ be the spherical catenoid obtained by rotating
$\sigma_a\subset\B^2_{+}$ (see \eqref{eq:half_space} for the definition
of $\B^2_{+}$) about the $u$-axis, where $\sigma_a$ is the catenary given by
\eqref{eq:parametric equation of catenary} and $a>0$ is the hyperbolic
distance between $\sigma_{a}$ and the origin. Mori, Do Carmo and Dajczer,
B{\'e}rard and Sa Earp, and Seo proved the following result
(see \cite{Mor81,dCD83,BSE10,Seo11}):
There exist two constants $A_1\approx{}0.46288$ and
$A_{2}\approx{}0.5915$ %=\frac{1}{2}\cosh^{-1}\left(\sqrt{\frac{11+8\sqrt{2}}{7}}\right)
such that $\CC_a$ is unstable if $0<a<A_1$, and
$\CC_a$ is globally stable if $a>A_2$.

\begin{remark}
The constants $A_1$ and $A_2$ were given by Seo
in \cite[Corollary 4.2]{Seo11} and by B{\'e}rard and Sa Earp
in \cite[Lemma 4.4]{BSE10} respectively.
A few years ago, Do Carmo and Dajczer showed that $\CC_a$
is unstable if $a\lessapprox{}0.42315$ in \cite{dCD83},
and Mori showed that $\CC_a$ is stable if
$a>\cosh^{-1}(3)\approx{}1.76275$
in \cite{Mor81} (see also \cite[p.34]{BSE09}).
\end{remark}

According to the numerical computation, B{\'e}rard and Sa Earp
claimed that $A_1$ should be the same as $A_2$
(see \cite[Proposition 4.10]{BSE10}). In this paper,
we will prove their claim.
More precisely, we have the following theorem.

\begin{theorem}[\cite{BSE10}]\label{thm:BSE theorem}
There exists a constant $a_c\approx{}0.49577$ such that
the following statements are true:
\begin{enumerate}
     \item $\CC_a$ is an unstable minimal surface with Morse index
           one if $0<a<a_c$;
     \item $\CC_a$ is a globally stable minimal surface if
           $a\geq{}a_c$.
  \end{enumerate}
\end{theorem}

\begin{remark}As we will see, the constant $a_{c}$ is the
unique critical number of the function $\varrho(a)$
given by \eqref{eq:Gomes function I}.
\end{remark}
%$a_c\approx{}0.49577389$

Similar to the case of hyperbolic and parabolic catenoids,
we want to know whether the globally stable spherical catenoids
are least area minimal surfaces.
In this paper, we prove that there exists a positive number
$a_l$ given by \eqref{eq:lambda_l} such that $\CC_a$ is a least
area minimal surface if $a\geq{}a_l$. More precisely, we will
prove the following result.

\begin{theorem}\label{thm:main theorem}
There exists a constant $a_l\approx{}1.10055$ defined by
\eqref{eq:lambda_l} such that for any $a\geq{}a_l$ the
catenoid $\CC_a$ is a least area
minimal surface in the sense of Meeks and Yau.
\end{theorem}

Mori \cite{Mor82} and Do Carmo and Dajczer \cite{dCD83} also
studied the helicoid in the hyperboloid model of the hyperbolic
$3$-space $\H^3$. Roughly speaking, a helicoid $\Hcal_{\abar}$ in the
upper half space model of the hyperbolic $3$-space $\H^3$ could be
obtained by rotating the (upper) semi unit circle with center
the origin along the $t$-axis about angle $\abar{}v$ and translating
it along the $t$-axis about \emph{hyperbolic distance} $v$ for all
$v\in\R$ (see $\S$\ref{subsec:helicoids-upper-half model}).

Mori \cite{Mor82} studied the stability of the helicoids
$\Hcal_{\abar}$ for $\abar\geq{}0$ in the hyperbolic $3$-space $\H^3$.
He showed that $\Hcal_{\abar}$ is globally stable
if $\abar\leq{}3\sqrt{2}/4$, and it is unstable if
$\abar\geq\sqrt{105\pi}/8$. In this paper we will prove the
following result.

\begin{theorem}\label{thm:main theoremII}
For a family of minimal helicoids $\{\Hcal_{\abar}\}_{\abar\geq{}0}$
in the hyperbolic $3$-space $\HH^3$ that is defined by
\eqref{eq:helicoid in hyperboloid model},
there exist a constant $\abar_c=\coth(a_c)\approx{}2.17968$ such that
the following statements are true:
  \begin{enumerate}
    \item $\Hcal_{\abar}$ is a globally stable minimal surface
          if $0\leq{}\abar\leq{}\abar_c$, and
    \item $\Hcal_{\abar}$ is an unstable minimal surface with
          index infinity if $\abar>\abar_c$.
  \end{enumerate}
\end{theorem}

%\subsection*{Outline of the proofs of the theorems}
\noindent\textbf{Outline of the proofs of the theorems.}
On Theorem~\ref{thm:BSE theorem}, because of \eqref{eq:E=d0'} and
Theorem~\ref{thm:BSE-stability}, which
was proved by B{\'e}rard and Sa Earp in \cite{BSE10}, we just need to
show that the function $\varrho'(a)$ defined by (4.16) has a unique zero.
By Lemma 7.1 and Lemma 7.2, we can show that $\varrho'(a)$ is positive
around $a=0$, decreasing on $(0,0.71555]$, and negative on
$[0.53064,\infty)$, which can imply Theorem~\ref{thm:BSE theorem}.

On Theorem~\ref{thm:main theorem}, we consider any annulus-type compact
subdomain $\Sigma$ of a spherical catenoid $\CC_a$, we will show that
if $a\geq{}a_{l}$, then the area of $\Sigma$ is less than the area of
any disks bounded by $\partial\Sigma$, and $\Sigma$ is also the least
area annulus among all annuli with the same boundary $\partial\Sigma$.

On Theorem~\ref{thm:main theoremII}, since each helicoid is conjugate
to a catenoid (in one of the three types), they have the same stability.
We know the stability of all catenoids (by Theorem~\ref{thm:BSE theorem},
and the facts that hyperbolic and parabolic catenoids are always stable),
so we know the stability of all helicoids.

\vskip 0.45cm
%\subsection*{Plan of the paper}
\noindent\textbf{Plan of the paper.}
This paper is organized as follows.
\begin{itemize}
  \item In $\S{}$\ref{sec:prelim} we introduce three types of catenoids
        in the hypebloid model $\HH^3$ of the hyperbolic $3$-space, and
        the helicoids in three models $\HH^3$, $\B^3$ and $\U^3$ of the
        hyperbolic $3$-space.
  \item In $\S${}\ref{sec:existence of spherical catenoids} we define
        the spherical catenoids in $\B^3$, then we prove
        a theorem of Gomes (Theorem \ref{thm:Gomes1987-prop3.2-a}).
  \item In $\S${}\ref{sec:stable catenoid} we introduce
        Jacobi fields on the catenoids
        (following B{\'e}rard and Sa Earp in \cite{BSE10}) and
        prove Theorem~\ref{thm:BSE theorem}.
  \item In $\S${}\ref{sec:least area catenoid} we prove
        Theorem~\ref{thm:main theorem}.
  \item In $\S${}\ref{sec:stable helicoid} we prove
        Theorem{}\ref{thm:main theoremII}.
  \item In $\S${}\ref{sec:technical lemmas}, we list and prove some technical
        lemmas which are used to prove Theorem~\ref{thm:BSE theorem} and
        Theorem~\ref{thm:main theorem}.
\end{itemize}

%===================================================================
\section{Preliminaries}\label{sec:prelim}

\subsection{Basic theory of minimal surfaces}\label{subsec:minimal surfaces}

Let $\Sigma$ be a surface immersed
in a $3$-dimensional Riemannian manifold
$M^{3}$. We pick up a local orthonormal frame field
$\{e_1,e_2,e_3\}$ for $M^{3}$ such that, restricted to $\Sigma$,
the vectors $\{e_1,e_2\}$ are tangent to $\Sigma$ and the
vector $e_3$ is perpendicular to $\Sigma$. Let
$A=(h_{ij})_{2\times{}2}$ be the second fundamental
form of $\Sigma$, whose entries $h_{ij}$ are represented by
\begin{equation*}
   h_{ij}=\inner{\nabla_{e_i}e_{3}}{e_{j}}\ ,
   \quad{}i,j=1,2\ ,
\end{equation*}
where $\nabla$ is the covariant derivative in $M^{3}$, and
$\inner{\cdot}{\cdot}$ is the metric of $M^{3}$.
An immersed surface $\Sigma\subset{}M^{3}$ is called a
\emph{minimal surface} if its \emph{mean curvature}
$H=h_{11}+h_{22}$ is identically zero.

For any immersed minimal surface $\Sigma$
in $M^{3}$, the {\em Jacobi operator} on $\Sigma$ is
defined as follows
\begin{equation}\label{eq:Jacobi operator}
   \Lcal=\Delta_{\Sigma}+(|A|^2+\Ric(e_3))\ ,
\end{equation}
where $\Delta_{\Sigma}$ is the Lapalican on $\Sigma$,
$|A|^2=\sum_{i,j=1}^{2}h_{ij}^2$ is the square of the
the length of the second fundamental form on $\Sigma$ and
$\Ric(e_3)$ is the Ricci curvature of $M^{3}$
in the direction $e_3$.

Suppose that $\Sigma$ is a complete minimal surface immersed in
a complete Riemannian $3$-manifold $M^{3}$.
For any compact connected subdomain $\Omega$ of $\Sigma$,
its first eigenvalue is defined by
\begin{equation}\label{eq:1st eigenvalue of Omega}
   \lambda_{1}(\Omega)=\inf\left\{-\int_{\Omega}f\Lcal{}f
   \ \left|\ f\in{}C_{0}^\infty(\Omega)\ \text{and}\
   \int_{\Omega}f^2=1\right.\right\}\ .
\end{equation}
We say that $\Omega$ is \emph{stable} if $\lambda_{1}(\Omega)>0$,
\emph{unstable} if $\lambda_{1}(\Omega)<0$ and
\emph{maximally weakly stable} if $\lambda_{1}(\Omega)=0$.

\begin{lemma}\label{lem:monotonicity of eigenvalue}
Suppose that $\Omega_1$ and $\Omega_2$ are connected
subdomains of $\Sigma$ with $\Omega_1\subset\Omega_2$, then
\begin{equation*}
  \lambda_{1}(\Omega_1)\geq\lambda_{1}(\Omega_2)\ .
\end{equation*}
If $\Omega_2\setminus\overline{\Omega}_1\ne\emptyset$, then
\begin{equation*}
  \lambda_{1}(\Omega_1)>\lambda_{1}(\Omega_2)\ .
\end{equation*}
\end{lemma}

\begin{remark}If $\Omega\subset\Sigma$ is maximally weakly
stable, then for any compact connected subdomains
$\Omega_1,\Omega_2\subset\Sigma$ satisfying
$\Omega_1\subsetneq\Omega\subsetneq\Omega_2$, we have that
$\Omega_1$ is stable whereas $\Omega_2$ is unstable.
\end{remark}

Let $\Omega_1\subset\Omega_2\subset\cdots\subset\Omega_n\subset\cdots$
be an exhaustion of $\Sigma$, then the first eigenvalue of
$\Sigma$ is defined by
\begin{equation}\label{eq:1st eigenvalue of Sigma}
   \lambda_{1}(\Sigma)=\lim_{n\to\infty}\lambda_{1}(\Omega_n)\ .
\end{equation}
This definition is independent of the choice of the exhaustion.
We say that $\Sigma$ is \emph{globally stable} or
\emph{stable} if $\lambda_{1}(\Sigma)>0$ and
\emph{unstable} if $\lambda_{1}(\Sigma)<0$.

The following theorem was proved by Fischer-Colbrie and Schoen in
\cite[Theorem 1]{FCS80} (see also \cite[Proposition 1.39]{CM11}).

\begin{theorem}[Fischer-Colbrie and Schoen]\label{thm:FCS80}
Let $\Sigma$ be a complete two-sided minimal surface in a
Riemannian $3$-manifold $M^{3}$, then $\Sigma$ is stable
if and only if there exists a positive function
$\phi:\Sigma\to\R$ such that $\Lcal{}\phi=0$.
\end{theorem}

The \emph{Morse index} of a compact connected subdomain
$\Omega$ of $\Sigma$ is the number of negative eigenvalues of the
Jacobi operator $\Lcal$ (counting with multiplicity) acting on the
space of smooth sections of the normal bundle that vanishes on
$\partial\Omega$. The \emph{Morse index} of $\Sigma$ is the
supremum of the Morse indices of compact subdomains of $\Sigma$.

The following proposition of Fischer-Colbrie can be applied to
show that some unstable minimal surface has infinite Morse index.

\begin{theorem}[{\cite[Proposition 1]{Fis85}}]\label{prop:FC85}
Let $\Sigma$ be a complete two-sided minimal surface in a
Riemannian $3$-manifold $M^{3}$. If $\Sigma$ has finite Morse index then
there is a compact set $K$ in $\Sigma$ so that
$\Sigma\setminus{}K$ is stable and there exists a positive function $\phi$ on
$\Sigma$ so that $\Lcal{}\phi=0$ on $\Sigma\setminus{}K$.
\end{theorem}

%\subsection{Least area minimal annuli in the sense of Meeks and Yau}
%\label{subsec:meek-yau least area}

Suppose that $\Sigma$ is a complete minimal surface immersed in
a complete Riemannian $3$-manifold $M^{3}$. For any compact subdomain
$\Omega$ of $\Sigma$, it is said to be \emph{least area} if its
area is smaller than that of any other surface in the same
homotopic class with the same boundary as $\partial\Omega$. We
say that $\Sigma$ is a \emph{least area minimal surface} if
any compact subdomain of $\Sigma$ is least area.

Let $S$ be a compact annulus-type minimal surface immersed in a
Riemannian $3$-manifold $M^{3}$. Suppose that the boundary of $S$
is the union of two simple closed curves $C_1,C_2$ which bound two
least area minimal disks $D_1,D_2$ respectively. The annulus $S$ is
called a \emph{least area minimal surface in the sense of Meeks and Yau}
in $M^{3}$ if
\begin{enumerate}
  \item $\Area(S)\leq\Area(S')$ for each annulus $S'$ with
        $\partial{}S'=\partial{}S$, and
  \item $\Area(S)<\Area(D_1)+\Area(D_2)$,
\end{enumerate}
where $\Area(\cdot)$ denotes the area of the surfaces in $M^{3}$
(see \cite[p.~412]{MY1982(t)}). A complete annulus-type minimal
surface $\Sigma$ immersed in $M^{3}$ is called a
\emph{least area minimal surface in the sense of Meeks and Yau}
if any annulus-type compact subdomain of $\Sigma$, which is
homotopically equivalent to $\Sigma$,
is a least area minimal surface in the sense of Meeks and Yau.

\subsection{Models of the hyperbolic $3$-space}\label{subsec:hyperbolic space models}

In this paper, we work in three models of the hyperbolic $3$ space:
the hyperboloid model $\HH^3$, the Poincar{\'e} ball model $\B^3$ and the upper
half space model $\U^3$ (see \cite[$\S$A.1]{BP92}).

\subsubsection{Hyperboloid model $\HH^3$}

We consider the Lorentzian $4$-space
$\Lbb^{4}$, i.e. a vector space $\R^{4}$ with the Lorentzian
inner product
\begin{equation}\label{eq:Lorentzian inner product}
   \inner{x}{y}=-x_{1}y_{1}+x_{2}y_{2}+x_{3}y_{3}+x_{4}y_{4}
\end{equation}
where $x,y\in\R^{4}$. Its isometry group is $\SO^{+}(1,3)$.
The hyperbolic space $\HH^{3}$ can be
considered as the unit sphere of $\Lbb^{4}$:
\begin{equation}
   \HH^{3}=\{x\in\Lbb^{4}\ |\ \inner{x}{x}=-1,\,x_{1}\geq{}1\}\ .
\end{equation}

\subsubsection{Poincar{\'e} ball model $\B^3$}

The The Poincar{\'e} ball model $\B^{3}$ of the hyperbolic $3$-space is
the open unit ball
\begin{equation*}
   \B^{3}=\{(u,v,w)\in\R^{3}\ | \ u^{2}+v^{2}+w^{2}<{}1\},
\end{equation*}
equipped with the hyperbolic metric
\begin{equation*}
   ds^{2}=\frac{4(du^{2}+dv^{2}+dw^{2})}{(1-r^{2})^{2}}\ ,
\end{equation*}
where $r=\sqrt{u^{2}+v^{2}+w^{2}}$. The orientation preserving
isometry group of $\B^3$ is denoted by $\Mob(\B^3)$, which
consists of M\"obius transformations that preserve the unit
ball $\B^{3}$ (see \cite[Theorem 1.7]{MT98}).
The hyperbolic space $\B^{3}$ has a natural compactification:
$\overline{\B^{3}}=\B^{3}\cup{}S_{\infty}^{2}$,
where $S_{\infty}^{2}\cong\C\cup\{\infty\}$ is called the
\emph{Riemann sphere}.

\subsubsection{Upper-half space model $\U^3$}

Consider the upper half space model of hyperbolic $3$-space,
i.e., a three dimensional space
\begin{equation*}
   \U^{3}=\{z+tj\ |\ z\in\C\ \text{and}\ t>0\}\ ,
\end{equation*}
which is equipped with the (hyperbolic) metric
\begin{equation*}
   ds^{2}=\frac{|dz|^{2}+dt^{2}}{t^{2}}\ ,
\end{equation*}
where $z=x+iy$ for $x,y\in\R$. The orientation preserving
isometry group of $\U^3$ is denoted by $\PSL_{2}(\C)$, which
consists of linear fractional transformations.

%---------------------------------------------------------------

\subsection{Three types of catenoids in the hyperboloid model $\HH^3$}
\label{subsec:catenoids in hyperboloid model}

We follow do Carmo and Dajczer \cite{dCD83} to describe three types
of catenoids in $\HH^3$. For any subspace $P$ of the Lorentzian
$4$-space $\Lbb^4$, let $\O(P)$ be the subgroup of $\SO^{+}(1,3)$
which leaves $P$ pointwise fixed.

\begin{definition}\label{def:rotation surface}
Let $\{e_1,\ldots,e_4\}$ be an orthonormal basis of
$\Lbb^4$ (it may not be the standard orthonormal basis). Suppose
that $P^2=\Span\{e_3,e_4\}$, $P^3=\Span\{e_1,e_3,e_4\}$ and
$P^3\cap\HH^3\ne\emptyset$. Let $C$ be a regular curve in
$P^3\cap\HH^3=\HH^2$ that does not meet $P^2$.
The orbit of $C$ under the action of $\O(P^2)$ is called a
\emph{rotation surface} \emph{generated by $C$ around $P^2$}.
\end{definition}

If a rotation surface in Definition \ref{def:rotation surface}
has mean curvature zero, then it's called a \emph{catenoid} in
$\HH^3$. There are three types of catenoids
in $\HH^3$: spherical catenoids, hyperbolic catenoids, and
parabolic catenoids.

\subsubsection{Spherical catenoids}
\label{subsub:spherical catenoild-lorentz}
The spherical catenoid is obtained as follows.
Let $\{e_1,\ldots,e_4\}$ be an orthonormal basis of $\Lbb^4$ such
that $\inner{e_4}{e_4}=-1$. Suppose that $P^2$, $P^3$ and $C$ are the
same as those defined in Definition \ref{def:rotation surface}.
For any point $x\in\Lbb^4$, write $x=\sum{}x_{k}e_{k}$. If
the curve $C$ is parametrized by
\begin{equation}\label{eq:spherical catenoid-x1(s)}
   x_{1}(s)=\sqrt{\atilde\cosh(2s)-1/2}\ ,
   \quad \atilde>1/2\,
\end{equation}
and
\begin{equation}\label{eq:spherical catenoid-x3(s)-x4(s)}
   x_{3}(s)=\sqrt{x_{1}^{2}(s)+1}\,\sinh(\phi(s))\ ,\
   x_{4}(s)=\sqrt{x_{1}^{2}(s)+1}\,\cosh(\phi(s))\ ,
\end{equation}
where
\begin{equation}
   \phi(s)=\int_{0}^{s}\frac{\sqrt{\atilde{}^2-1/4}}
   {(\atilde\cosh(2\sigma)+1/2)
   \sqrt{\atilde\cosh(2\sigma)-1/2}}\,d\sigma\ ,
\end{equation}
then the rotation surface, denoted by $\Cscr_{1}(\atilde)$,
is a complete minimal surface in $\HH^3$,
which is called a \emph{spherical catenoid}.

\subsubsection{Hyperbolic catenoids}
\label{subsub:hyperbolic catenoild-lorentz}
The hyperbolic catenoid is obtained as follows.
Let $\{e_1,\ldots,e_4\}$ be an orthonormal basis of $\Lbb^4$
such that $\inner{e_1}{e_1}=-1$. Suppose that $P^2$, $P^3$ and $C$ are the
same as those defined in Definition \ref{def:rotation surface}.
For any point $x\in\Lbb^4$, write $x=\sum{}x_{k}e_{k}$. If
the curve $C$ is parametrized by
\begin{equation}\label{eq:hyperbolic catenoid-x1(s)}
   x_{1}(s)=\sqrt{\atilde\cosh(2s)+1/2}\ ,
   \quad \atilde>1/2\,
\end{equation}
and
\begin{equation}\label{eq:hyperbolic catenoid-x3(s)-x4(s)}
   x_{3}(s)=\sqrt{x_{1}^{2}(s)-1}\,\sin(\phi(s))\ ,\
   x_{4}(s)=\sqrt{x_{1}^{2}(s)-1}\,\cos(\phi(s))\ ,
\end{equation}
where
\begin{equation}
   \phi(s)=\int_{0}^{s}\frac{\sqrt{\atilde{}^2-1/4}}
   {(\atilde\cosh(2\sigma)-1/2)
   \sqrt{\atilde\cosh(2\sigma)+1/2}}\,d\sigma\ ,
\end{equation}
then the rotation surface, denoted by $\Cscr_{-1}(\atilde)$,
is a complete minimal surface in $\HH^3$,
which is called a \emph{hyperbolic catenoid}.

\subsubsection{Parabolic catenoids}
\label{subsub:parabolic catenoild-lorentz}
The parabolic catenoid is obtained as follows.
Let $\{e_1,\ldots,e_4\}$ be a pseudo-orthonormal basis of $\Lbb^4$ such
that $\inner{e_1}{e_1}=\inner{e_3}{e_3}=0$, $\inner{e_1}{e_3}=-1$ and
$\inner{e_j}{e_k}=\delta_{jk}$ for $j=2,4$ and $k=1,2,3,4$
(see \cite[P. 689]{dCD83}).
Suppose that $P^2$, $P^3$ and $C$ are the
same as those defined in Definition \ref{def:rotation surface}.
For any point $x\in\Lbb^4$, write $x=\sum{}x_{k}e_{k}$. If
the curve $C$ is parametrized by
\begin{equation}\label{eq:parabolic catenoid-x1(s)}
   x_{1}(s)=\sqrt{\cosh(2s)}\ ,
\end{equation}
and
\begin{equation}\label{eq:parabolic catenoid-x3(s)-x4(s)}
   x_{4}(s)=x_{1}(s)\int_{0}^{s}\frac{d\sigma}
   {\sqrt{\cosh^3(2\sigma)}}\ ,\
   x_{3}(s)=-\frac{1+x_{4}^2(s)}{2x_{1}(s)}\ ,
\end{equation}
then the rotation surface, denoted by $\Cscr_{0}$,
is a complete minimal surface in $\HH^3$, which is called a
\emph{parabolic catenoid}.
Up to isometries, the parabolic catenoid $\Cscr_{0}$ is unique
(see \cite[Theorem (3.14)]{dCD83}).

%--------------------------------------------------------------------
\subsection{Helicoids in the three models of the hyperbolic $3$-space}
\label{sec:helicoid-II}

In this subsection, we introduce the equations of helicoids in the
hyperbloid model of the hyperbolic $3$-space at first.
In order to visualize the helicoids in the hyperbolic $3$-space,
we will study the parametric equations of helicoids in the
Poincar{\'e} ball model and upper half space model of
the hyperbolic $3$-space.
To derive the formulas in \eqref{eq:helicoid--ball model}
and \eqref{eq:helicoid--upper half},
we apply the isometries from the hyperboloid model to the
Pioncar{\'e} ball model and the upper half space model
(see \cite[$\S$A.1]{BP92}).

\subsubsection{Helicoids in $\HH^3$}

The helicoid $\Hcal_{\abar}$ in the hyperbloid model $\HH^3$ of
the hyperbolic $3$-space is the surface parametrized by the
$(u,v)$-plane in the following way (see \cite[p.699]{dCD83}):
\begin{equation}\label{eq:helicoid in hyperboloid model}
   \Hcal_{\abar}=\left\{x\in\HH^{3}\ \left|
   \begin{aligned}
      &x_{1}=\cosh{}u\cosh{}v, &&x_{2}=\cosh{}u\sinh{}v\\
      &x_{3}=\sinh{}u\cos(\abar{}v), &&x_{4}=\sinh{}u\sin(\abar{}v)
   \end{aligned}
   \right.\right\}\ ,
\end{equation}
where $-\infty<u,v<\infty$.
For any constant $\abar\geq{}0$, the helicoid $\Hcal_{\abar}\subset\HH^3$ is an
embedded minimal surface (see \cite{Mor82}). In the hyperboloid model
$\HH^3$, the axis of the helicoid $\Hcal_{\abar}$ is given by
\begin{equation*}
   (\cosh{}v,\sinh{}v,0,0)\ ,
   \quad{}-\infty<v<\infty\ ,
\end{equation*}
which is the intersection of the $x_{1}x_{2}$-plane and $\HH^3$.

\subsubsection{Helicoids in $\B^3$}

The helicoid $\Hcal_{\abar}$ in the Poincar{\'e} ball model $\B^3$
of the hyperbolic $3$-space is given by
(see Figure~\ref{fig:helicoid in B3})
\begin{equation}\label{eq:helicoid--ball model}
  \Hcal_{\abar}=\left\{(x,y,z)\in\B^{3}\ \left|\
     \begin{aligned}
        x&=\frac{\sinh{}u{}\cos(\abar{}v)}{1+\cosh{}u{}\cosh{}v}\ , \\
        y&=\frac{\sinh{}u{}\sin(\abar{}v)}{1+\cosh{}u{}\cosh{}v}\ , \\
        z&=\frac{\cosh{}u{}\sinh{}v}{1+\cosh{}u{}\cosh{}v}
     \end{aligned}\right.\right\}\ ,
\end{equation}
where $-\infty<u,v<\infty$.

%==============================================================================
\begin{figure}[htbp]
  \centering
  \includegraphics[scale=0.15]{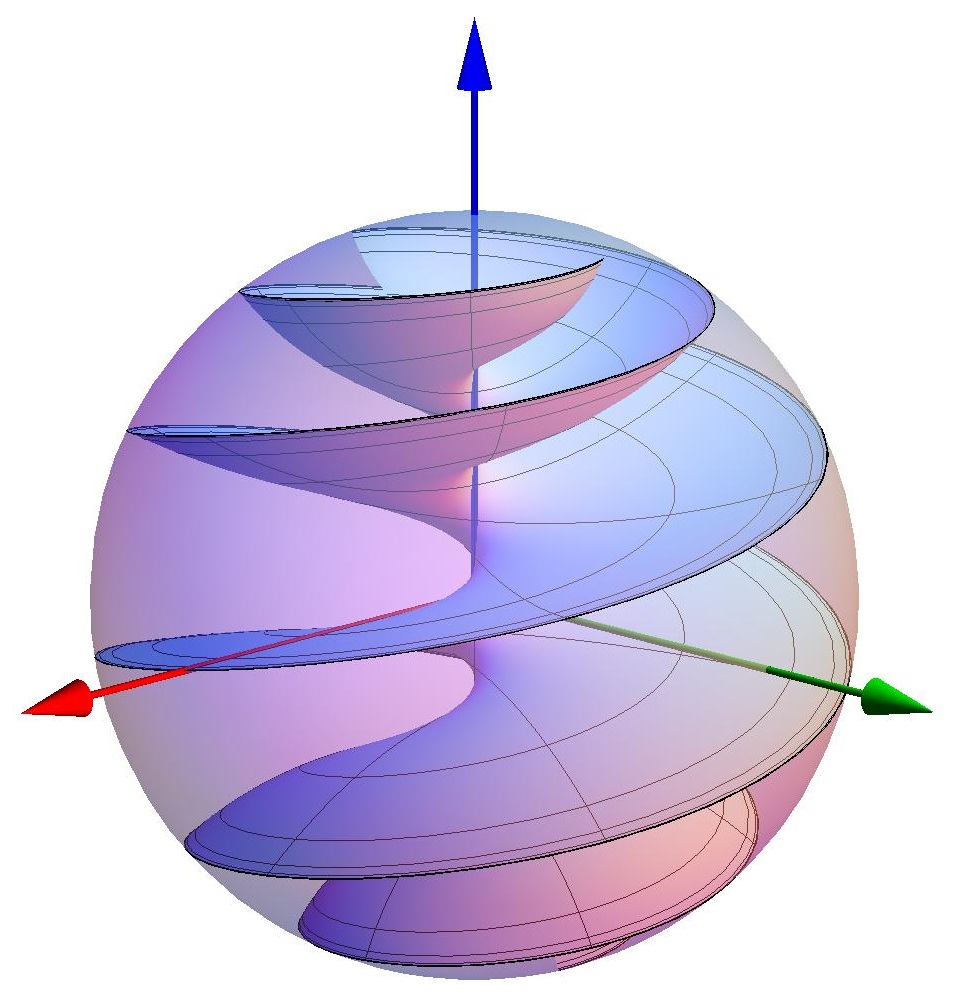}
  \caption{The helicoid $\Hcal_{\abar}$ with $\abar=5$ in the Poincar{\'e} ball
  model of hyperbolic space. The curves perpendicular
  to the spirals are geodesics in $\B^3$.}\label{fig:helicoid in B3}
\end{figure}
%===============================================================================

\subsubsection{Helicoids in $\U^3$}
\label{subsec:helicoids-upper-half model}
The helicoid $\Hcal_{\abar}$ in the upper half space model $\U^3$
of the hyperbolic $3$-space is given by
(see Figure \ref{fig:helicoid in H3})
\begin{equation}\label{eq:helicoid--upper half}
  \Hcal_{\abar}=\{(z,t)\in\U^{3}\ |\ z=e^{v+\sqrt{-1}\,\abar{}v}\tanh{}u\
     \text{and}\ t=e^{v}\sech{}u\}\ ,
\end{equation}
where $-\infty<u,v<\infty$,
and the axis of $\Hcal_{\abar}$ is the $t$-axis.

From the equation \eqref{eq:helicoid--upper half}, we can see that each
helicoid $\Hcal_{\abar}$ is invariant under the one-parameter group $G_{\abar}\subset\PSL_{2}(\C)$ consisting of loxodromic transformations
which fix the same $t$-axis , i.e.
\begin{equation*}
   G_{\abar}=\left\{z\mapsto{}\exp\left(v+\sqrt{-1}\,\abar{}v\right)z \ | \
   -\infty<v<\infty\right\}\ .
\end{equation*}
When $v=0$, from \eqref{eq:helicoid--upper half} we get a semi unit circle
in the $xt$-plane, whose center is the origin.

%========================================================================
\begin{figure}[htbp]
  \centering
  \includegraphics[scale=0.19]{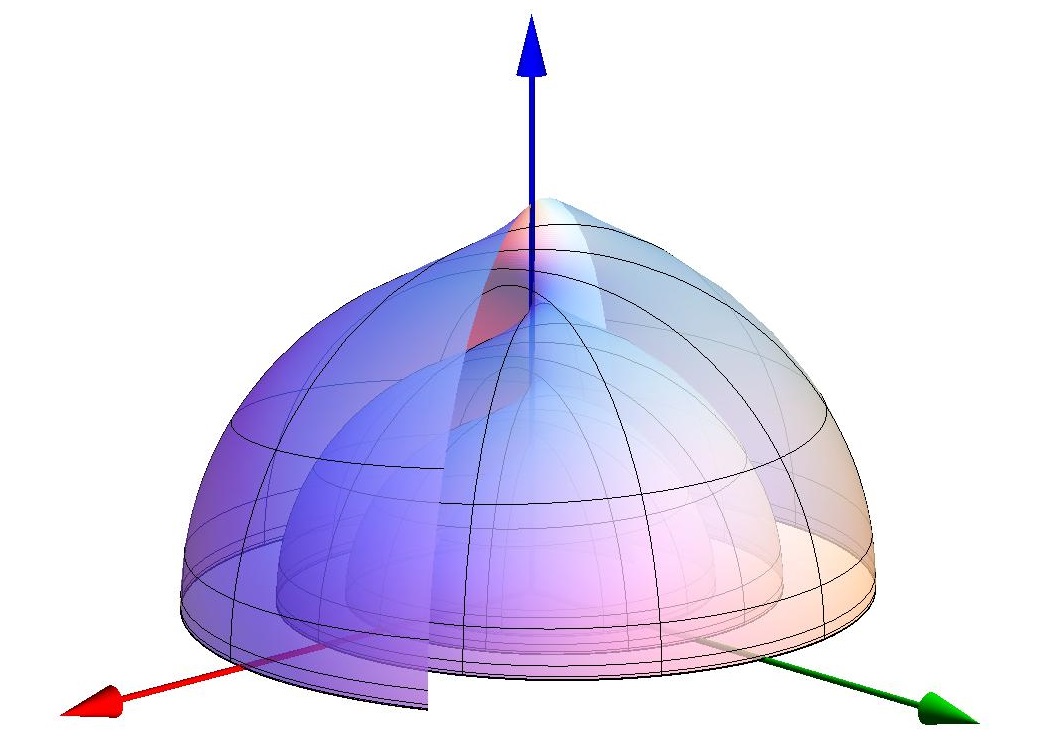}
  \caption{The helicoid $\Hcal_{\abar}$ with $\abar=5$ in the upper half
  space model of hyperbolic space $\H^3$. The curves perpendicular
  to the spirals are geodesics in $\H^3$.}\label{fig:helicoid in H3}
\end{figure}
%========================================================================

%-----------------------------------------------------------
\section{Existence and Uniqueness of Spherical Catenoids}
\label{sec:existence of spherical catenoids}

In this section, we will prove a theorem of Gomes, which plays an
important role in the paper \cite{Wang11a} for constructing barrier
surfaces.

%----------------------------------------------------------------------

\subsection{Spherical catenoids in $\B^3$}
\label{subsec:catnoids}

In this subsection, we follow Hsiang (see \cite{Hsi82,BCH09}) to
introduce the minimal spherical catenoids in $\B^3$.
Let $X$ be a subset of $\B^{3}$, we define
the {\em asymptotic boundary} of $X$ by
\begin{equation}
   \partial_{\infty}X=\overline{X}
   \cap{}S_{\infty}^{2}\ ,
\end{equation}
where $\overline{X}$ is the closure of $X$ in
$\overline{\B}{}^{3}$.

Using the above notation, we have
$\partial_\infty\B^3=S_{\infty}^2$. If $P$ is a geodesic plane
in $\B^{3}$, then $P$ is perpendicular to $S_{\infty}^{2}$ and
$C\stackrel{\text{def}}{=}\partial_{\infty}P$ is
an Euclidean circle on $S_{\infty}^{2}$. We also say that
$P$ is {\em asymptotic to} $C$.

Suppose that $G\cong\SO(3)$ is a subgroup of $\Mob(\B^{3})$ that leaves
a geodesic $\gamma\subset\B^{3}$ pointwise fixed. We call $G$ the
\emph{spherical group} of $\B^{3}$ and $\gamma$ the
\emph{rotation axis} of $G$. A surface in $\B^{3}$ which is invariant
under $G$ is called a \emph{spherical surface} or a
\emph{surface of revolution}.
%(see Figure~\ref{fig:catenoid in B3}).
For two circles $C_{1}$
and $C_{2}$ in $\B^{3}$, if there is a geodesic $\gamma$,
such that each of $C_{1}$ and $C_{2}$ is invariant under the
group of rotations that fixes $\gamma$ pointwise, then $C_{1}$
and $C_{2}$ are said to be \emph{coaxial}, and $\gamma$ is
called the {\em rotation axis} of $C_{1}$ and $C_{2}$.

\subsubsection{The warped product metric}
Suppose that $G$ is the spherical group of $\B^{3}$ along
the geodesic
\begin{equation}\label{eq:rotation axis}
   \gamma_{0}=\{(u,0,0)\in\B^3\ |\ -1<u<1\}\ ,
\end{equation}
then $\B^3/G\cong\B_{+}^2$, where
\begin{equation}\label{eq:half_space}
   \B_{+}^2=\{(u,v)\in\B^2\ |\ v\geq{}0\}\ .
\end{equation}
We shall equip the half space $\B_{+}^2$ with a warped product metric.

For any point $p=(u,v)\in\B_{+}^2$, there is a unique
geodesic segment $\gamma'$ passing through $p$ that is
perpendicular to $\gamma_{0}$ at $q$.
Let $x=\dist(O,q)$ and $y=\dist(p,q)=\dist(p,\gamma_{0})$
(see Figure~\ref{fig:intrinsic metric}), where
$\dist(\cdot,\cdot)$ denotes the hyperbolic distance, then by
\cite[Theorem 7.11.2]{Bea95}, we have
\begin{equation}\label{eq:(x,y) in terms of (u,v)}
   \tanh{}x=\frac{2u}{1+(u^2+v^2)}
   \quad\text{and}\quad
   \sinh{}y=\frac{2v}{1-(u^2+v^2)}\ .
\end{equation}
Equivalently, we also have
\begin{equation}\label{eq:(u,v) in terms of (x,y)}
   u=\frac{\sinh{}x\cosh{}y}{1+\cosh{}x\cosh{}y}
   \quad\text{and}\quad
   v=\frac{\sinh{}y}{1+\cosh{}x\cosh{}y}\ .
\end{equation}

It's well known that $\B_{+}^2$ can be equipped with
the \emph{metric of warped product} in terms
of the parameters $x$ and $y$ as follows:
\begin{equation}\label{eq:warped product metric}
   ds^2=\cosh^{2}y\cdot{}dx^2+dy^2\ ,
\end{equation}
where $dx$ represents the hyperbolic metric on the geodesic
$\gamma_0$ in \eqref{eq:rotation axis}. We call the horizontal
geodesic $\{(u,0)\in\B_{+}^2\ |\ -1<u<1\}$ the $x$-axis and
the vertical geodesic $\{(0,v)\in\B_{+}^2\ |\ 0\leq{}v<1\}$
the $y$-axis. The orientations of the $x$-axis and the $y$-axis
are considered to be the same as that of the $u$-axis and the
$v$-axis respectively. Thus we also consider that the $x$-axis
and the $y$-axis are equivalent to the $u$-axis and the $v$-axis
respectively.

\begin{definition}\label{def:catenoid without parameter}
If $\CC$ is a minimal surface of revolution in $\B^3$ with
respect to the axis $\gamma_{0}$ in \eqref{eq:rotation axis},
then it is called a \emph{catenoid}
and the curve $\sigma=\CC\cap\B_{+}^{2}$ is called the
\emph{generating curve} of $\CC$ or a \emph{catenary}.
\end{definition}

%====================================================================
\begin{figure}[htbp]
  \begin{center}
     \includegraphics[scale=1]{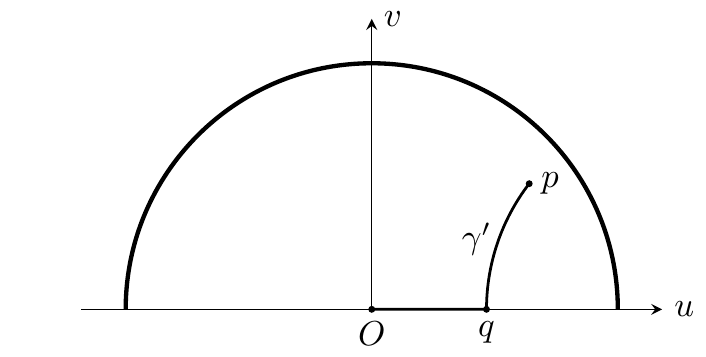}
   \end{center}
  \caption{For a point $p$ in $\B_{+}^2$ with the warped product
   metric, its coordinates $(x,y)$ are defined by $x=\dist(O,q)$
   and $y=\dist(p,q)$.}\label{fig:intrinsic metric}
\end{figure}
%====================================================================

\subsubsection{Arc length parametrization of a catenary}
\label{sec:catenary}

Let $\sigma\subset\B_{+}^{2}$ be the generating curve of a
minimal catnoid $\CC$. Suppose that the parametric equations
of $\sigma$ are given by:
$x=x(s)$ and $y=y(s)$, where $s\in(-\infty,\infty)$ is an
arc length parameter of $\sigma$. By the argument in
\cite[pp. 486--488]{Hsi82}, the curve $\sigma$ satisfies the
following equations
\begin{equation}\label{eq:differential equations of Pi}
   \frac{2\pi\sinh{}y\cdot\cosh^{2}y}
   {\sqrt{\cosh^2{}y+(y')^2}}=
   2\pi\sinh{}y\cdot\cosh{}y\cdot\sin\theta=k\
   (\text{constant})\ ,
\end{equation}
where $y'=dy/dx$ and $\theta$ is the angle between the tangent
vector of $\sigma$ and the vector $e_y=\partial/\partial{}y$
at the point $(x(s),y(s))$ (see Figure~\ref{fig:diff eq of sigma}).

By the argument in \cite[pp.54--58]{Gom87}), up to isometry,
we assume that the curve $\sigma$ is only symmetric about the $y$-axis
(by assumption it's the same as the $v$-axis)
and intersects the $y$-axis orthogonally at $y_0=y(0)$,
and so $y'(0)=0$.
Substitute these to \eqref{eq:differential equations of Pi},
we get $k=2\pi\sinh(y_0)\cosh(y_0)$, and then we have
the following equation
\begin{equation}\label{eq:angle-alpha}
   \sin\theta=\frac{\sinh(y_0)\cosh(y_0)}{\sinh(y)\cosh(y)}
             =\frac{\sinh(2y_0)}{\sinh(2y)}\ .
\end{equation}
Now we solve for $dx/dy$ in terms of $y$
in \eqref{eq:differential equations of Pi} and integrate
$dx/dy$ from $y_0$ to $y$ for any $y\geq{}y_0$, then we have
the following equality
\begin{equation}\label{eq: catenary equation}
   x(y)=\int_{y_0}^{y}\frac{\sinh(2y_0)}{\cosh{}y}
             \frac{dy}{\sqrt{\sinh^2(2y)-\sinh^2(2y_0)}}\ .
\end{equation}
Let $y\to\infty$ in \eqref{eq: catenary equation},
we get (see Figure~\ref{fig:Gomes distance})
\begin{equation}\label{eq:Gomes distance}
   x(\infty)=\int_{y_0}^{\infty}\frac{\sinh(2y_0)}{\cosh{}y}
             \frac{dy}{\sqrt{\sinh^2(2y)-\sinh^2(2y_0)}}\ .
\end{equation}
We can see that the hyperbolic distance between two totally geodesic
planes bounded by the boundary of the catenoid $\CC$, whose generating
curve is $\sigma$, is equal to the double of $x(\infty)$. We will
see that this distance can not be too large, actually it's
around $1$ (see Theorem~\ref{thm:Gomes1987-prop3.2-a}).

%===================================================================
\begin{figure}[htbp]
  \begin{center}
     \includegraphics[scale=0.9]{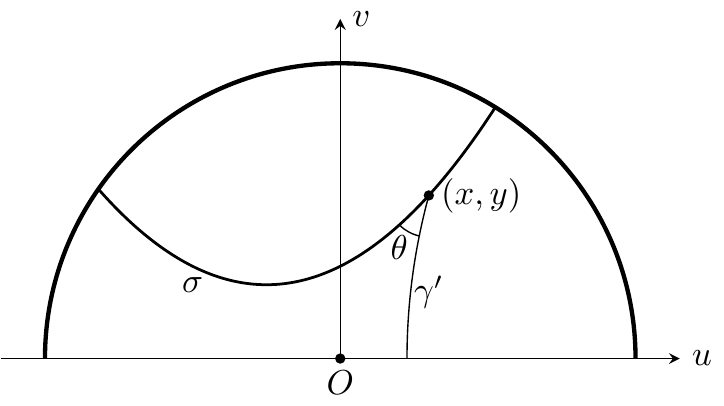}
   \end{center}
  \caption{$\theta$ is the angle between the parametrized
  curve $\sigma$ and the geodesic $\gamma'$ at the point
  $(x(s),y(s))\in\sigma\cap\gamma'$, where $\gamma'$ is
  perpendicular to the $u$-axis.}\label{fig:diff eq of sigma}
\end{figure}
%===================================================================

Replacing the initial data $y_0$ by a parameter $a\in(0,\infty)$ in
\eqref{eq: catenary equation}, and let $\sigma_a\subset\B_{+}^2$ be the
catenary which is symmetric about the $y$-axis and whose initial
data is $a$ (actually the hyperbolic distance between $\sigma_{a}$
and the origin of $\B_{+}^2$ is equal to $a$).

In order to get the parametric equation of the catenary
$\sigma_a$ with arc length, we define a function of the variable $a$
by rewriting \eqref{eq: catenary equation}:
\begin{equation}\tag{\ref{eq: catenary equation}$'$}
                \label{eq:catenary(a,t)}
  \rho(a,t)=\int_{a}^{t}\frac{\sinh(2a)}{\cosh\tau}
  \frac{d\tau}{\sqrt{\sinh^2(2\tau)-\sinh^2(2a)}}\ ,
  \quad{}t\geq{}a\ .
\end{equation}
Recall that the semi disk $\B_{+}^2$ is equipped with the metric
\eqref{eq:warped product metric}, it's easy to get
the arc length of the catenary $\sigma_a$:
\begin{equation}\label{eq:arc length of catenary}
  s(a,t)
   =\int_{a}^{t}\frac{\sinh(2\tau)}
   {\sqrt{\cosh^2(2\tau)-\cosh^2(2a)}}\,d\tau
   =\frac{1}{2}\,\cosh^{-1}\left(\frac{\cosh(2t)}{\cosh(2a)}\right)\ ,
\end{equation}
where $t\geq{}a$. For any $s\in(-\infty,\infty)$, let
\begin{align}\label{eq:x(a,s)}
  x(a,s)&=\sqrt{2}\,\sinh(2a)
          \int_{0}^{s}\frac{\sqrt{\cosh(2a)\cosh(2t)-1}}
          {\cosh^{2}(2a)\cosh^{2}(2t)-1}\,dt\ ,\\
          \label{eq:y(a,s)}
  y(a,s)&=a+\int_{0}^{s}\frac{\cosh(2a)\sinh(2t)}
           {\sqrt{\cosh^{2}(2a)\cosh^{2}(2t)-1}}\,dt\\
           \label{eq:y(a,s)-II}
        &=\frac{1}{2}\,\cosh^{-1}(\cosh^{2}(2a)\cosh^{2}(2s))\ .
\end{align}
It's easy to verify that
\begin{equation}\label{eq:x(a,s) and rho(a,t)}
   x(a,s)=\rho(a,y(a,s))
\end{equation}
for $s\geq{}0$ and that the map
\begin{equation}\label{eq:parametric equation of catenary}
   s\mapsto(x(a,s),y(a,s))
\end{equation}
is arc-length parametrization of the catenary $\sigma_a$
for $s\in(-\infty,\infty)$, where $\rho(\cdot,\cdot)$ is
given by \eqref{eq:catenary(a,t)}
(see \cite[Proposition 4.2]{BSE10}).

%===================================================================
\begin{figure}[htbp]
  \begin{center}
     \includegraphics[scale=0.9]{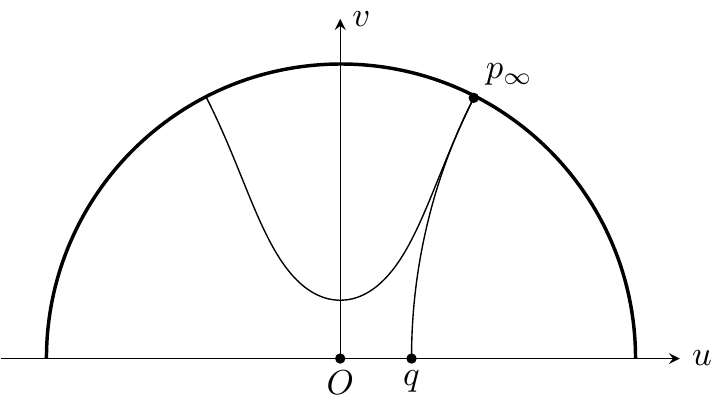}
  \end{center}
  \caption{The distance $x(\infty)$ defined in
  \eqref{eq:Gomes distance} is equal to $\dist(O,q)$, where
  the point $q$ is the intersection of the $u$-axis and the
  unique geodesic which is perpendicular to both the $u$-axis
  at $q$ and $\partial_{\infty}\B_{+}^2$ at $p_\infty$ (here
  $p_\infty$ is one of the asymptotic boundary points of
  $\sigma$ given by \eqref{eq:parametric equation of catenary}).
  In this figure, $y_{0}=0.4$, and so
  $x(\infty)\approx{}0.49268$.}
  \label{fig:Gomes distance}
\end{figure}
%====================================================================

\subsubsection{Parametrization of a catenoid}
%We define the catenoid with a parameter at first.

\begin{definition}\label{def:sphereical catenoid with parameter}
The surface of revolution along the axis $\gamma_{0}$ in
\eqref{eq:rotation axis} generated by the catenary $\sigma_a$
\eqref{eq:parametric equation of catenary} is called a
\emph{catenoid}, and is denoted by $\CC_{a}$ for $0<a<\infty$.
\end{definition}

Next we shall find the parametric equation of the catenoid $\CC_{a}$.
Recall that the generating curve $\sigma_a$ of the catenoid $\CC_{a}$ is
parametrized by arc length: $x=x(a,s)$ and $y=y(a,s)$ given by
\eqref{eq:x(a,s)}, \eqref{eq:y(a,s)} and \eqref{eq:y(a,s)-II}.
Just as in \eqref{eq:(u,v) in terms of (x,y)}, we define
\begin{equation*}
   u(a,s)=\frac{\sinh{}x\cosh{}y}{1+\cosh{}x\cosh{}y}
         \quad\text{and}\quad
   v(a,s)=\frac{\sinh{}y}{1+\cosh{}x\cosh{}y}\ .
\end{equation*}
The parametric equation of the catenoid
$\CC_a$ in $\B^3$ is given by
\begin{equation}\label{eq:position vector of catenary}
  Y(a,s,\theta)=
     \begin{pmatrix}
        u\\
        v\omega_\theta
     \end{pmatrix}\ ,
     \quad
     s\in\R\ \text{and}\
     \theta\in[0,2\pi]\ ,
\end{equation}
where $\omega_\theta=\begin{pmatrix}
        \cos\theta\\
        \sin\theta
     \end{pmatrix}$.
Direct computation shows that the unit normal vector
of the catenoid $\CC_a$ at $Y(a,s,\theta)$ is
\begin{equation}
  N(a,s,\theta)=
     \begin{pmatrix}
        v_{s}\\
        -u_{s}\omega_\theta
     \end{pmatrix}\ ,
     \quad
     s\in\R\ \text{and}\
     \theta\in[0,2\pi]\ ,
\end{equation}
where $u_s$ and $v_s$ are the partial derivatives
of $u$ and $v$ on $s$ respectively.

\begin{remark}The spherical catenoid $\Cscr_{1}(\atilde)$
in the hyperboloid model $\HH^3$ is isometric to the
spherical catenoid $\CC_a$ in the Poincar{\'e} ball model
$\B^3$ if and only if $2\atilde=\cosh(2a)$ (see
Lemma~\ref{lem:relation between spherical catenoids}).
\end{remark}

\subsection{The theorem of Gomes}

Obviously the asymptotic boundary of any spherical catenoid
$\CC_a$ is the union of two circles (see also
\cite[Proposition 3.1]{Gom87}). It's important for us
to determine whether there exists a minimal spherical
catenoid asymptotic to any given pair of disjoint circles on
$S_{\infty}^2$, since in \cite{Wang11a} we
construct quasi-Fuchsian $3$-manifolds
which contain arbitrarily many incompressible minimal surface
by using the (least area) minimal spherical catenoids as the
barrier surfaces.

If $C_{1}$ and $C_{2}$ are two disjoint circles on
$S_{\infty}^{2}$, then they are always coaxial. In fact,
let $P_{1}$ and $P_{2}$ be the geodesic planes asymptotic
to $C_{1}$ and $C_{2}$ respectively, there always exists
a unique geodesic $\gamma$ such that $\gamma$ is perpendicular
to both $P_{1}$ and $P_{2}$. Therefore $C_{1}$ and $C_{2}$ are
coaxial with respect to $\gamma$.
We may define the distance between $C_{1}$ and $C_{2}$ by
\begin{equation}
   d_{L}(C_1,C_2)=\dist(P_1,P_2)\ .
\end{equation}
In order to prove Theorem \ref{thm:Gomes1987-prop3.2-a},
we need define a function $\varrho(a)$ of the parameter $a$.
Let $t$ approach infinity in \eqref{eq:catenary(a,t)},
we define the function
\begin{equation}\label{eq:Gomes function I}
   \varrho(a)=\rho(a,\infty)=\int_{a}^{\infty}
             \frac{\sinh(2a)}{\cosh{}t}
             \frac{dt}{\sqrt{\sinh^2(2t)-\sinh^2(2a)}}\ .
\end{equation}
Using the substitution $t\mapsto{}t+a$, the above function
\eqref{eq:Gomes function I} can be written as
\begin{equation}\tag{\ref{eq:Gomes function I}$'$}
                \label{eq:Gomes function II}
   \varrho(a)=\int_{0}^{\infty}\frac{\sinh(2a)}{\cosh(a+t)}
              \frac{dt}{\sqrt{\sinh^2(2a+2t)-\sinh^2(2a)}}\ .
\end{equation}

\begin{theorem}[Gomes]\label{thm:Gomes1987-prop3.2-a}
There exists a constant $a_{c}\approx{}0.49577$
such that for two disjoint circles
$C_{1},C_{2}\subset{}S_{\infty}^{2}$, if
\begin{equation*}
   d_{L}(C_{1},C_{2})\leq{}2\varrho(a_{c})\approx{}1.00229\ ,
\end{equation*}
then there exist a spherical minimal catenoid $\CC$ which is
asymptotic to $C_{1}\cup{}C_{2}$, where $\varrho(a)$ is the function
defined by \eqref{eq:Gomes function I}.
\end{theorem}

\begin{remark}In \cite[p. 402]{dOS98}, de Oliveria and Soret
show that for any two congruent circles
(in $\partial_{\infty}\U^3=\R^2\times\{0\}$)
of Euclidean diameter $d$ and disjoint from each other by the
Euclidean distance $D$, there exists \emph{two} catenoids bounding the
two circles if and only if $D/d\leq\delta$ for some $\delta>0$.
Direct computation shows that
$\delta=\cosh(\varrho(a_c))-1\approx{}0.12763$, where $\varrho(a)$ is
the function defined by \eqref{eq:Gomes function I} and $a_c$ is the
unique critical number of the function $\varrho(a)$.
\end{remark}

%===================================================================
\begin{figure}[htbp]
  \begin{center}
     \includegraphics[scale=1]{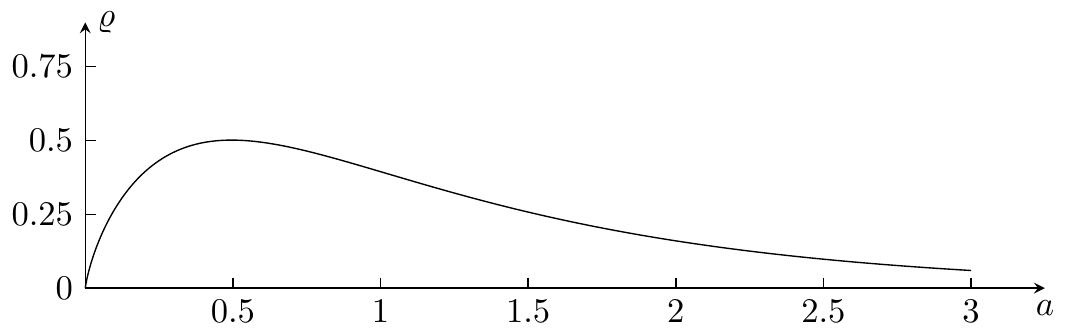}
  \end{center}
  \caption{The graph of the function $\varrho(a)$ defined by
  \eqref{eq:Gomes function I} for
  $a\in[0,3]$. It seems that $\varrho(a)$ only has a unique critical number.}
  \label{fig:Gomes function}
\end{figure}
%===================================================================

\begin{proof}[\bf{Proof of Theorem~\ref{thm:Gomes1987-prop3.2-a}}]
Let $\varrho(a)$ be the function defined by
\eqref{eq:Gomes function I} or \eqref{eq:Gomes function II}.
We claim that $\varrho(0)=0$, and as $a$ increases $\varrho(a)$
increases monotonically, reaches a maximum, then decreases
asymptotically to zero as $a$ goes to infinity
(see also \cite[Proposition 3.2]{Gom87} and
Figure~\ref{fig:Gomes function}).

It's easy to show $\varrho(a)\to{}0$ as $a\to\infty$.
In fact, we have
\begin{align*}
   \varrho(a)
      &=\int_{0}^{\infty}\frac{\sinh(2a)}{\cosh(a+t)}
        \frac{dt}{\sqrt{\sinh^2(2a+2t)-\sinh^2(2a)}}\\
      &=\int_{0}^{\infty}\frac{1}{\cosh(t+a)}
        \frac{dt}{\sqrt{\left(\dfrac{\sinh(2a+2t)}
        {\sinh(2a)}\right)^2-1}}\\
      &<\int_{0}^{\infty}\frac{1}{\cosh{}a}
        \frac{dt}{\sqrt{(\sinh(2t)+\cosh(2t))^2-1}}\ .
\end{align*}
Since $\sinh(2t)+\cosh(2t)=e^{2t}$, we have
\begin{equation}\label{eq:estimate on d0}
\begin{aligned}
   \varrho(a)
       &<\frac{1}{\cosh{}a}\int_{0}^{\infty}
         \frac{dt}{\sqrt{e^{4t}-1}}
        =\frac{1}{\cosh{}a}\int_{0}^{\infty}
         \frac{e^{-2t}}{\sqrt{1-e^{-4t}}}\,dt\\
       &=\frac{1}{\cosh{}a}\cdot\frac{\pi}{4}
         \longrightarrow{}0\quad\text{as}\ a\to\infty\ .
\end{aligned}
\end{equation}
Besides, since $\displaystyle\lim_{a\to{}0^{+}}\varrho(a)=0$,
$\varrho(a)>0$ for $a\in(0,\infty)$
and $\varrho(a)\to{}0$ as $a\to\infty$,
it must have at least one  maximum value in $(0,\infty)$.

By the argument in the proof of Theorem~\ref{thm:BSE theorem}
in $\S$\ref{sec:stable catenoid},
we know that $\varrho'(a)$ has a unique zero $a_c$ such that
$\varrho'(a)>0$ if $0<a<a_c$ and $\varrho'(a)<0$ if $a>a_c$, hence
the proof of the claim is complete.
According to the numerical computation: the function
$\varrho(a)$ achieves its (unique) maximum value
$\approx{}0.501143$ when $a=a_c\approx{}0.49577$,
and so $2\varrho(a_c)\approx{}1.00229$.
\end{proof}

Theorem \ref{thm:Gomes1987-prop3.2-a} shows the existence
of spherical minimal catenoids.
On the other hand, we also have the uniqueness of catenoids
in the sense of following theorem proved by Levitt and
Rosenberg (see \cite[Theorem 3.2]{LR85} and
\cite[Theorem 3]{dCGT86}).
Recall that a complete minimal surface $\Sigma$ of $\H^3$ is
\emph{regular at infinity} if $\partial_\infty\Sigma$
is a $C^2$-submanifold of $S^2_{\infty}$ and
$\overline\Sigma=\Sigma\cup\partial_\infty\Sigma$ is a
$C^2$-surface (with boundary) of $\overline{\H^3}$.

\begin{theorem}[Levitt and Rosenberg]
Let $C_1$ and $C_2$ be two disjoint round circles on
$S_{\infty}^{2}$ and let $\CC$ be a connected minimal surface
immersed in $\H^3$ with $\partial_{\infty}\CC=C_1\cup{}C_2$
and $\CC$ regular at infinity. Then $\CC$ is a
spherical catenoid.
\end{theorem}

%======================================================================
\section{Stability of Spherical Catenoids} % stability
\label{sec:stable catenoid}

In this section we will prove Theorem \ref{thm:BSE theorem}.
Let $\Sigma$ be a complete minimal surface immersed in a
complete Riemannian $3$-manifold $M^{3}$, and let $\Omega$ be any
subdomain of $\Sigma$.
Recall that a \emph{Jacobi field} on $\Omega \subset \Sigma$
is a $C^\infty$ function $\phi$ such that $\Lcal\phi = 0$
on $\Omega$.

According to Theorem~\ref{thm:FCS80}, in order to show that
a complete minimal surface $\Sigma\subset{}M^{3}$ is stable,
we just need to find a positive Jacobi field on $\Sigma$.
On the other hand, if a Jacobi field on $\Sigma$ changes its
sign between the interior and the exterior of a compact subdomain
$\Omega$ of $\Sigma$ and vanishes on $\partial\Omega$,
we can conclude that $\Omega$ is a maximally weakly stable
minimal surface, which also implies that $\Sigma$ is unstable.

The geometry of the ambient space provides useful Jacobi fields.
More precisely, we have the following classical results.

\begin{theorem}[{\cite[pp. 149--150]{Xin03}}]
\label{thm:Killing Jacobi field}
Let $\Sigma$ be a complete minimal surface immersed in
complete Riemannian $3$-manifold $M^{3}$
and let $V$ be a Killing field on $M^{3}$. The function
$\zeta = \inner{V}{N}$, given by the inner product in
$M^{3}$ of the Killing field $V$ with the unit normal $N$ to
the immersion, is a Jacobi field on $\Sigma$.
\end{theorem}

\begin{theorem}[{\cite[Theorem 2.7]{BdC80b}}]
\label{thm:BdC}
Let $X(a,\cdot):\Sigma\to{}M^{3}$ be a $1$-parameter family of
minimal immersions. Then, for each fixed $a_0$, the function
\begin{equation*}
  \xi=\biginner{\ppl{X}{a}(a_0,\cdot)}{N}
\end{equation*}
is a Jacobi field on $X(a_0,\Sigma)$, where
$\inner{\cdot}{\cdot}$ is the inner product in $M^{3}$ and
$N$ is the unit normal vector field on the minimal surface
$X(a_0,\Sigma)$.
\end{theorem}

\subsection{Jacobi fields on spherical catenoids}
Next we will follow B{\'e}rard and Sa Earp
\cite{BSE10} to introduce the vertical Jacobi fields and
the variation Jacobi fields on the minimal spherical catenoids
$\{\CC_a\}_{a>0}$ in $\B^3$, which
will be used to prove Theorem~\ref{thm:BSE theorem}.

\begin{definition}Let $V$ be the Killing vector field associated
with the hyperbolic translations along the geodesic
$t\mapsto(\tanh(t/2),0,0)\in\B^3$.
The \emph{vertical Jacobi field} on the
catenoid $\CC_a$ is the function
\begin{equation}\label{eq:vertical Jacobi field}
  \zeta(a,s)=\inner{V(a,s,\theta)}{N(a,s,\theta)}\ ,
\end{equation}
where $V(a,s,\theta)$ is the restriction of the Killing vector field
$V$ to the minimal catenoid $\CC_a$ defined by
\eqref{eq:position vector of catenary}.

The \emph{variation Jacobi field} on the
catenoid $\CC_a$ is
\begin{equation}\label{eq:variation Jacobi field}
  \xi(a,s)=-\inner{Y_a(a,s,\theta)}{N(a,s,\theta)}\ ,
\end{equation}
where $Y_a=\ppl{Y}{a}$.
\end{definition}

In order to find the detail expressions of the vertical
and the variation Jacobi fields on the catenoids, we need
some notations (see \cite[$\S$4.2]{BSE10}). Let
\begin{equation}\label{eq:def of f(a,s)}
  f(a,s)=\frac{\sinh^{2}(2a)\cosh(2s)}{\cosh^{2}(2a)\cosh^{2}(2s)-1}\ ,
\end{equation}
and let
\begin{equation}\label{eq:def of I(a,t)}
  I(a,t)=\frac{n(\cosh(2a),\cosh(2t))}
  {d(\cosh(2a),\cosh(2t))}
\end{equation}
where
\begin{itemize}
  \item $n(A,T)=A(3-A^2)T^2+(A^2-1)T-2A$, and
  \item $d(A,T)=(AT+1)^2(AT-1)^{3/2}$.
\end{itemize}
For the functions $x(a,s)$ and $y(a,s)$ given by \eqref{eq:x(a,s)}
and \eqref{eq:y(a,s)}, the notations
$x_a$, $x_s$, $y_a$ and $y_s$ denote the partial derivatives of
$x(a,s)$ and $y(a,s)$ on $a$ and $s$
respectively.

\begin{proposition}[{\cite[$\S$4.2.1]{BSE10}}]
\label{prop:vertical and variation fields}
The vertical Jacobi field $\zeta(a,s)$ is given by
\begin{equation}\label{eq:vertical field}
  \zeta(a,s)
     = \sqrt{2}\,\cosh(y(a,s))y_{s}(a,s)
     = \frac{\cosh(2a)\sinh(2s)}{\sqrt{\cosh(2a)\cosh(2s)-1}}\ .
\end{equation}
The variation Jacobi field $\xi(a,s)$ is given by
\begin{align}\label{eq:variation fiels I}
  \xi(a,s)
     &=-\cosh(y(a,s))(x_a(a,s)y_s(a,s)-x_s(a,s)y_a(a,s)) \\
     \label{eq:variation fiels II}
     &=f(a,s)-\zeta(a,s)\int_{0}^{s}I(a,t)dt\ ,
\end{align}
where $f(a,s)$ and $I(a,t)$ are given by
\eqref{eq:def of f(a,s)} and \eqref{eq:def of I(a,t)}
respectively.
\end{proposition}

%---------------------------------------------------------

Since $x(a,\infty)$ is well defined for any $a>0$, we may set
\begin{equation}
  E(a)=\ddl{}{a}\,x(a,\infty)
      =\sqrt{2}\int_{0}^{\infty}I(a,t)dt\ .
\end{equation}
Equivalently we have the following identity
(see \cite[p. 3665]{BSE10}):
\begin{equation}\label{eq:E=d0'}
  E(a)=\frac{\varrho'(a)}{\sqrt{2}}\ ,
\end{equation}
where $\varrho'(a)$ is defined by \eqref{eq:derivative of d0}, which
is derivative of the function $\varrho(a)$
given by \eqref{eq:Gomes function I}.

For any (connected) interval $\mathbf{I}\subset\R$, we define
\begin{equation}
  \CC(a,\mathbf{I})=\{Y(a,s,\theta)\in\B^3\ |\ s\in{}\mathbf{I}
  \ \text{and}\ \theta\in[0,2\pi]\}\ ,
\end{equation}
where $Y(a,s,\theta)$ is given by \eqref{eq:position vector of catenary}.
Obviously $\CC(a,\mathbf{I})\subset\CC_{a}$.

\begin{lemma}[{\cite[Lemma 4.5]{BSE10}}]
\label{lem:BSE-stability-half-catenoid}
For any constant $a>0$, the half catenoids $\CC(a,(-\infty,0])$
and $\CC(a,[0,\infty))$ are both stable.

Any Jacobi field $\eta(a, s)$ depending only on the radial
variable s on $\CC_a$ can change its sign at most
once on either $(-\infty,0]$ or $[0,\infty)$.
\end{lemma}

\begin{proof}The first part follows from the fact that $\zeta(a,s)$
doesn't change its sign on either $\CC(a,(-\infty,0])$ or
$\CC(a,[0,\infty))$ and $\Lcal{}\zeta=0$.

Assume that some Jacobi field $\eta(a,s)$ on $\CC_a$ changes its
sign more than once on $[0,\infty)$, then $\eta(a,s)$ has at least
two zeros on $[0,\infty)$, say
$0<{}z_1<z_2<\cdots$. Let $\mathbf{I}=[z_1,z_2]$ and let
$\phi(a,s)$ be the restriction of $\eta(a,s)$ to
$\CC(a,\mathbf{I})$, then we have
$\phi\in{}C_{0}^\infty(\CC(a,\mathbf{I}))$ and $\Lcal{}\phi=0$,
which imply that $\lambda_{1}(\CC(a,\mathbf{I}))\leq{}0$.
This is a contradiction, since $\CC(a,\mathbf{I})$ is a compact
connected subdomain of $\CC(a,[0,\infty))$, which must be stable.
\end{proof}

The following theorem, whose proof can be found
in \cite[p. 3663]{BSE10}, is crucial to the proof of
Theorem~\ref{thm:BSE theorem}. Because of \eqref{eq:E=d0'},
\eqref{eq:derivative of d0} and
Lemma~\ref{lem:derivative of d0}, we always have
$\cosh^{2}(2a)<3$ if $E(a)=0$, hence it is not necessary to
consider the case when $\cosh^{2}(2a)<3$ in the
the proof of \cite[Theorem 4.7 (1)]{BSE10}.

\begin{theorem}[{\cite[Theorem 4.7 (1)]{BSE10}}]
\label{thm:BSE-stability}
Let $\sigma_a$ be the catenary given by
\eqref{eq: catenary equation} and let $\CC_a$ be the
minimal surface of revolution along the $u$-axis
whose generating curve is the catenary $\sigma_a$.
\begin{enumerate}
  \item If $E(a)\leq{}0$, then $\CC_a$ is stable.
  \item If $E(a)>0$, then $\CC_a$ is unstable
        and has index $1$.
\end{enumerate}
\end{theorem}

\begin{proof}(1) As state in Lemma \ref{lem:BSE-stability-half-catenoid},
the function $\xi(a, s)$ can change its sign at most
once on $(0,\infty)$ and $(-\infty,0)$ respectively. Observe
that the function $\xi(a, s)$ is even and that $\xi(a,0)=1$. To determine
whether $\xi$ has a zero, it suffices to look at its behaviour at infinity.

If $E(a)<0$, then $\int_{0}^{\infty}I(a,t)dt<0$, which implies that
$\xi(a,s)\to\infty$ as $s\to\pm\infty$, therefore $\xi(a,s)>0$
for all $s\in(-\infty,\infty)$.

If $E(a)=0$, we have the following equation
\begin{equation}\label{eq:variation Jacobi field when E(a)=0}
  \xi(a,s)=f(a,s)+\zeta(a,s)\int_{s}^{\infty}I(a,t)dt\ .
\end{equation}
By \eqref{eq:E=d0'}, \eqref{eq:derivative of d0} and
Lemma~\ref{lem:derivative of d0}, we can see that if $E(a)=0$, then
\begin{equation*}
   \cosh^{2}(2a)\leq\cosh^{2}(2A_{3})=
   \left(\frac{1+\sqrt{5}}{2}\right)^2<3\ ,
\end{equation*}
and so $I(a,t)>0$ if $t$ is large enough.
If $s$ is sufficiently large, then $\xi(a,s)>0$ according to
\eqref{eq:variation Jacobi field when E(a)=0},
thus $\xi(a,s)>0$ for all $s\in(-\infty,\infty)$. Therefore
$\CC_{a}$ is stable if $E(a)\leq{}0$.

(2) Recall that the variation Jacobi field  $\xi(a,s)$ can
change its sign at most once on either $(0,\infty)$ or $(-\infty,0)$
by Lemma~\ref{lem:BSE-stability-half-catenoid}.
Now suppose that $E(a)>0$, since $\xi(a,0)=1$ and $\xi(a,s)\to-\infty$
as $s\to\pm\infty$, we know that
$\xi(a,s)$ has exactly two symmetric zeros in $(-\infty,\infty)$,
which are denoted by $\pm{}z(a)$.
Let $\CC(z(a))$ be the subdomain of $\CC_a$ defined by
\begin{equation}\label{eq:portion of Pi}
  \CC(z(a))=\CC(a,[-z(a),z(a)])\ .
\end{equation}
Let $\phi(a,s)$ be restriction of $\xi(a,s)$ to $\CC(z(a))$, then
$\phi\in{}C_{0}^\infty(\CC(z(a)))$ and $\Lcal{}\phi=0$. This implies
that $\lambda_{1}(\CC(z(a)))\leq{}0$, which can imply that any compact
connected subdomain of $\CC_a$ containing $\CC(z(a))$ must be
unstable by Lemma~\ref{lem:monotonicity of eigenvalue}.

Therefore $\CC_a$ has index at least one.
By \cite[Theorem 4.3]{Seo11} or \cite[$\S$3.3]{Tuz93},
$\CC_a$ has index one.
\end{proof}

\subsection{Proof of Theorem~\ref{thm:BSE theorem}}

Now we are able to prove Theorem~\ref{thm:BSE theorem}.

%===================================================================
\begin{figure}[htbp]
  \begin{center}
     \includegraphics[scale=0.8]{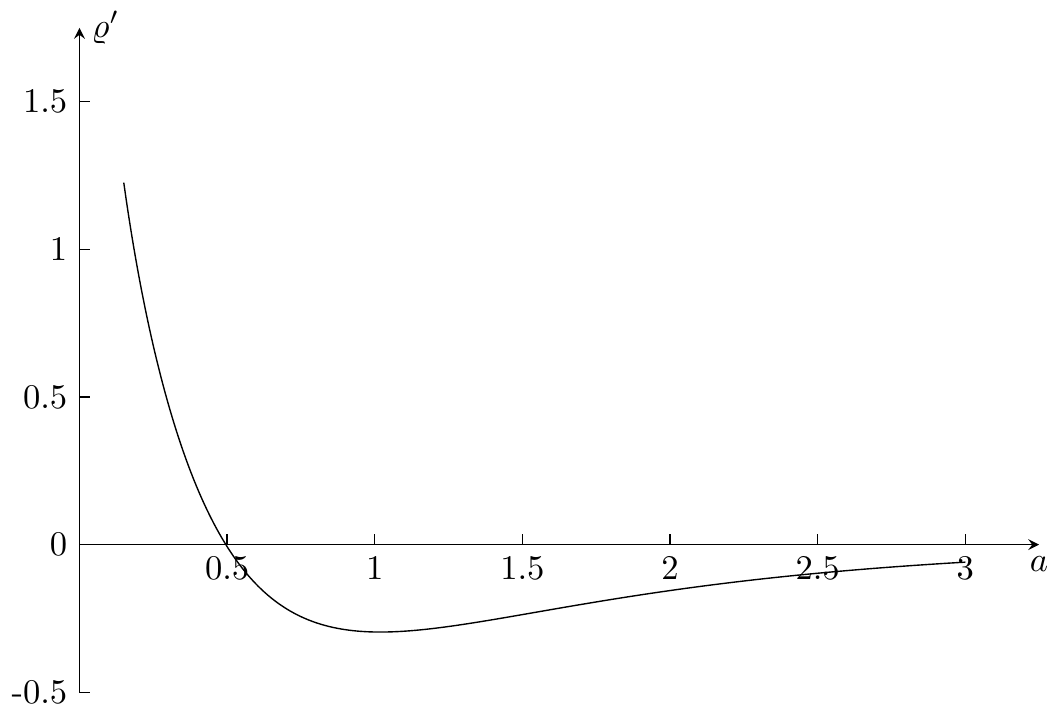}
  \end{center}
  \caption{The derivative of the function $\varrho(a)$ for
  $a\in(0,3]$. From the figure we can see that $\varrho'(a)$ has a unique zero,
  which is very close to $0.5$.}
  \label{fig:derivative of Gomes function}
\end{figure}
%===================================================================

\begin{BSEtheorem}
There exists a constant $a_c\approx{}0.49577$ such that
the following statements are true:
\begin{enumerate}
     \item $\CC_a$ is an unstable minimal surface with Morse index
           one if $0<a<a_c$;
     \item $\CC_a$ is a globally stable minimal surface if
           $a\geq{}a_c$.
  \end{enumerate}
\end{BSEtheorem}

\begin{proof}%[\bf{Proof of Theorem~\ref{thm:BSE theorem}}]
Recall that we have $E(a)=\varrho'(a)/\sqrt{2}$ by
\eqref{eq:E=d0'}. We claim that $\varrho'(a)$
satisfies the following conditions
(see Figure~\ref{fig:derivative of Gomes function}):
\begin{itemize}
   \item $\varrho'(a)\to\infty$ as $a\to{}0^{+}$ and
         $\varrho'(a)<0$ on $[A_3,\infty)$, and
   \item $\varrho'(a)$ is decreasing on $(0,A_4)$,
\end{itemize}
where $A_3<A_4$ are constants defined in
\eqref{eq:constant Lambda-3} and \eqref{eq:constant Lambda-4}.
These conditions can imply that $\varrho'$ has a unique zero
$a_c\in(0,\infty)$ such that $\varrho'(a)>0$ if $0<a<a_c$ and
$\varrho'(a)<0$ if $a>a_c$, hence together with
Theorem~\ref{thm:BSE-stability}, the theorem follows.

Next let's prove the above claim. It's easy to verify that
\begin{equation}\label{eq:derivative of d0}
   \varrho'(a)=\int_{0}^{\infty}
      \frac{\sinh(a+t)
      (5\cosh^2(a+t)-\cosh^{2}(3a+t))}
      {\cosh^2(a+t)\sqrt{\sinh(2t)\sinh^{3}(4a+2t)}}\,dt\ .
\end{equation}
By Lemma \ref{lem:derivative of d0}, $\varrho'(a)<0$ on
$(A_3,\infty)$. Now let
\begin{equation}
   h(a,t)=\frac{\sinh(a+t)
      (5\cosh^2(a+t)-\cosh^{2}(3a+t))}
      {\cosh^2(a+t)\sqrt{\sinh(2t)\sinh^{3}(4a+2t)}}\ .
\end{equation}
Then for any fixed constant $a>0$, we have the estimates
\begin{equation}
   h(a,t)\sim{}C_{1}(a)\left(
   \sqrt{\sinh{}t}+\frac{\cosh{}t}{\sqrt{\sinh{}t}}\right)\ ,
   \quad\text{as}\ t\to{}0\ ,
\end{equation}
and
\begin{equation}
   h(a,t)\sim
   \frac{C_{2}(a)}{\cosh{}t\cosh(2a+t)\sinh(4a+2t)}\ ,
   \quad\text{as}\ t\to\infty\ .
\end{equation}
Hence $\varrho'(a)$ is well defined for $a>0$, and then
\begin{align*}
   \lim_{a\to{}0^{+}}\varrho'(a)
    &=\int_{0}^{\infty}\frac{1}{\sinh{}t\cosh^{2}t}\,dt\\
    &=\left[\log\left(\frac{\cosh{}t-1}{\cosh{}t+1}\right)+
      \frac{1}{\cosh{}t}\right]_{t=0}^{t=\infty}=\infty\ .
\end{align*}
Next, we have
\begin{equation}\label{eq:2nd derivative of d0}
   \varrho''(a)=\int_{0}^{\infty}\frac{\psi(a,t)}
   {16\cosh^3(a+t)\sqrt{\sinh(2t)\sinh^{5}(4a+2t)}}\,dt\ ,
\end{equation}
where $\psi(a,t)$ is the function defined  by \eqref{eq:psi(t,l)}.
By the result in Lemma \ref{lem:second derivative of d[lambda]},
$\varrho''(a)<0$ for $a\in(0,A_4)$, thus
$\varrho'(a)$ is decreasing on $(0,A_4)$.
\end{proof}

%===================================================================
\section{Least Area Spherical Catenoids}
\label{sec:least area catenoid}

In this section, we will prove Theorem \ref{thm:main theorem}.
The following estimate is crucial to the proof of
Theorem \ref{thm:main theorem}.

\begin{lemma}\label{lem:area compare}
For all real numbers $a>0$, consider the functions
\begin{equation}\label{eq: area difference function}
   f(a)=\int_{a}^{\infty}\sinh{}t\cdot
   \left(\frac{\sinh(2t)}{\sqrt{\sinh^2(2t)-\sinh^2(2a)}}
   -1\right)dt\ ,
\end{equation}
and $g(a)=\cosh{}a-1$, then we have the following results:
\begin{enumerate}
   \item $f(a)$ is well defined for each fixed
         $a\in(0,\infty)$.
   \item $f(a)<g(a)$ for sufficiently large $a$.
\end{enumerate}
\end{lemma}

\begin{proof}(1) Using the substitution $t\to{}t+a$, we have
\begin{equation*}
   f(a)=\int_{0}^{\infty}\sinh(a+t)
   \left(\frac{\sinh(2a+2t)}{\sqrt{\sinh^2(2a+2t)-
   \sinh^2(2a)}}-1\right)dt\ .
\end{equation*}
We will prove that $f(a)<K\cosh{}a$, where
\begin{equation}\label{eq:limit of the ratio}
   K=\int_{0}^{1}\frac{1}{x^2}
          \left(\frac{1}{\sqrt{1-x^4}}-1\right)dx
\end{equation}
is a constant between $0$ and $1$.

Let $\displaystyle\Phi(a,t)=\frac{\sinh(2a+2t)}
{\sqrt{\sinh^2(2a+2t)-\sinh^2(2a)}}$, then
for any fixed $t\in[0,\infty)$, it's easy to
verify that $\Phi(a,t)$ is increasing on $[0,\infty)$
with respect to $a$. So
we have the estimate
\begin{align*}
   \Phi(a,t)   &\leq\lim_{a\to\infty}
                    \frac{\sinh(2a+2t)}
                    {\sqrt{\sinh^2(2a+2t)-\sinh^2(2a)}}\\
               &=\frac{\sinh(2t)+\cosh(2t)}{\sqrt{(\sinh(2t)+
                 \cosh(2t))^2-1}}\\
               &=\frac{e^{2t}}{\sqrt{e^{4t}-1}}
                =\frac{1}{\sqrt{1-e^{-4t}}}\ .
\end{align*}
Besides,
$\sinh(a+t)<(\sinh{}t+\cosh{}t)\cosh{}a=e^{t}\cosh{}a$,
therefore we have the following estimate
\begin{align*}
    f(a)
     &<\cosh{}a\int_{0}^{\infty}e^t\left(\frac{e^{2t}}
       {\sqrt{e^{4t}-1}}-1\right)dt\\
     &=\cosh{}a\int_{0}^{\infty}e^t\left(\frac{1}
       {\sqrt{1-e^{-4t}}}-1\right)dt\\
     &=\cosh{}a\int_{0}^{1}\frac{1}{x^2}
       \left(\frac{1}{\sqrt{1-x^4}}-1\right)dx
       \quad (t\mapsto{}x=e^{-t})
\end{align*}
Since $x^2+1\geq{}1$, we have
\begin{equation*}
\begin{aligned}
   K&=\int_{0}^{1}\frac{1}{x^2}\left(\frac{1}{\sqrt{1-x^4}}-1\right)dx\\
    &<\int_{0}^{1}\frac{1}{x^2}\left(\frac{1}{\sqrt{1-x^2}}-1\right)dx
     =1 \ ,
    \end{aligned}
\end{equation*}
where we use the substitution $x\to{}\sin{}x$ to evaluate the above
integral.

(2) We have proved that
$f(a)<K\cosh{}a$ for any $a\in[0,\infty)$. Let
\begin{equation}\label{eq:lambda_l}
   a_{l}=\cosh^{-1}\left(\frac{1}{1-K}\right)\ ,
\end{equation}
then $f(a)<g(a)$ if $a\geq{}a_l$.
\end{proof}

\begin{remark}The function $f(a)$ in
\eqref{eq: area difference function} has its geometric meaning:
$2\pi{}f(a)$ is the difference of the \emph{infinite area}
of one half of the catenoid $\CC_a$ and that of the annulus
\begin{equation*}
   \mathcal{A}=\{(0,v,w)\in\B^3\ |\ \tanh(a/2)
   \leq\sqrt{v^2+w^2}<1\}\ .
\end{equation*}
\end{remark}

\begin{remark}
The first definite integral in \eqref{eq:limit of the ratio} is an elliptic
integral. By the numerical computation, $K\approx{}0.40093$, and hence
$a_l\approx{}1.10055$.
\end{remark}

We need the coarea formula that will be used in the proof of
Theorem \ref{thm:main theorem}. The proof of
\eqref{eq:coarea formula I} in Lemma \ref{lem:coarea formula}
can be found in \cite{Wang10}.

\begin{lemma}[Calegari and Gabai\ {\cite[$\S$1]{CG06}}]
\label{lem:coarea formula}
Suppose $\Sigma$ is a surface in the hyperbolic $3$-space
$\B^3$. Let $\gamma\subset\B^3$ be a geodesic,
for any point $q\in\Sigma$, define $\theta(q)$ to be
the angle between the tangent space to $\Sigma$ at $q$, and
the radial geodesic that is through $q$
{\rm(}emanating from $\gamma${\rm)}
and is perpendicular to $\gamma$. Then
\begin{equation}\label{eq:coarea formula I}
   \Area(\Sigma\cap\Nscr_{s}(\gamma))=
   \int_{0}^{s}\int_{\Sigma\cap\partial\Nscr_{t}(\gamma)}
   \frac{1}{\cos\theta}\,dldt\ ,
\end{equation}
where $\Nscr_{s}(\gamma)$ is the hyperbolic $s$-neighborhood
of the geodesic $\gamma$.
\end{lemma}

%====================================================================
\begin{figure}[htbp]
  \begin{center}
     \includegraphics[scale=0.9]{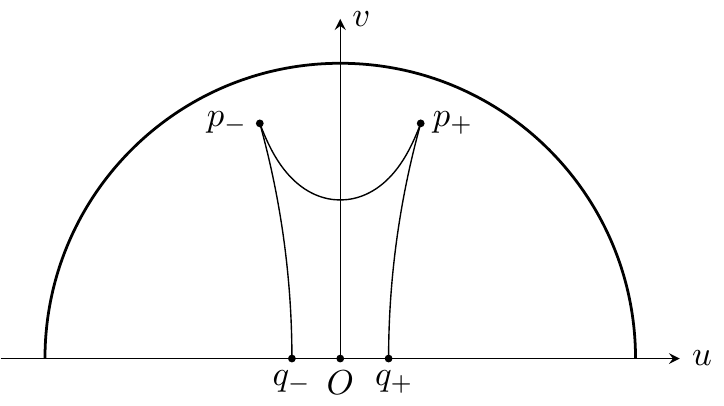}
  \end{center}
  \caption{The curve $\wideparen{p_{-}p_{+}}=\Sigma\cap\B_{+}^2$ is
  the portion of the catenary $\sigma_a$ with $a=1.2$.
  The curves $\wideparen{p_{+}q_{+}}=P_{+}\cap\B_{+}^2$ and
  $\wideparen{p_{-}q_{-}}=P_{-}\cap\B_{+}^2$.
  In this figure, $y_1=2.4$ and $x_1\approx{}0.330439$. By numerical
  computation: $\Area(\Sigma)=54.6636$ and
  $\Area(P_{+}\cup{}P_{-})=57.2643$.}
  \label{fig:least area catenoid--lambda=1.2}
\end{figure}
%====================================================================

Now we are able to prove Theorem \ref{thm:main theorem}.

\begin{maintheorem}
There exists a constant $a_l\approx{}1.10055$ defined by
\eqref{eq:lambda_l} such that for any
$a\geq{}a_l$ the catenoid $\CC_a$ is a least area
minimal surface in the sense of Meeks and Yau.
\end{maintheorem}

\begin{proof}%[\bf{Proof of Theorem \ref{thm:main theorem}}]
First of all, suppose that $a\geq{}a_{l}$ is an
arbitrary constant.
Suppose that $\partial_{\infty}\CC_{a}=C_1\cup{}C_2$, and
let $P_{i}$ be the geodesic plane asymptotic to $C_{i}$
($i=1,2$). Let $\sigma_{a}=\CC_a\cap\B_{+}^2$
be the generating curve of the
catenoid $\CC_a$.

For $x\in(-\varrho(a),\varrho(a))$, where $\varrho(a)$ is defined
by \eqref{eq:Gomes function I}, let $P(x)$ be the
geodesic plane perpendicular to the $u$-axis such
that $\dist(O,P(x))=|x|$. Now let
\begin{equation}\label{eq:symmetric compact annulus}
   \Sigma=\bigcup_{|x|\leq{}x_{1}}
   \left(\CC_{a}\cap{}P(x)\right)\ ,
\end{equation}
for some $0<x_{1}<\varrho(a)$.
Let $\partial\Sigma=C_{+}\cup{}C_{-}$. Note that $C_+$ and $C_-$ are
coaxial with respect to the $u$-axis. % or $\gamma_{0}$.

\begin{claim}\label{claim:one}
$\Area(\Sigma)<\Area(P_{+})+\Area(P_{-})$, where
$P_{\pm}$ are the compact subdomains of $P(\pm{}x_{1})$ that are
bounded by $C_{\pm}$ respectively
(see Figure~\ref{fig:least area catenoid--lambda=1.2}).
\end{claim}

\begin{proof}[\bf{Proof of Claim \ref{claim:one}}]
Recall that $P_{\pm}$ are two (totally) geodesic disks with
hyperbolic radius $y_{1}$, so the area of $P_{\pm}$ is given by
\begin{equation}
   \Area(P_{+})=\Area(P_{-})=4\pi\sinh^{2}\left(
   \frac{y_{1}}{2}\right)=2\pi(\cosh{}y_{1}-1)\ ,
\end{equation}
where $(x_1,y_1)\in\sigma_a$ satisfies the equation
\eqref{eq: catenary equation}.

Recall that
$\Area(\Sigma)=\Area(\Sigma\cap\Nscr_{y_{1}}(\gamma_{0}))$,
by the co-area formula we have
\begin{equation}
  \Area(\Sigma)=\int_{a}^{y_{1}}\left(
  \Length(\Sigma\cap\partial\Nscr_{t}(\gamma_{0}))\cdot
  \frac{1}{\cos\theta}\right)dt\ ,
\end{equation}
where the angle $\theta$ is given by \eqref{eq:angle-alpha}, hence
\begin{equation}
  \Area(\Sigma)=\int_{a}^{y_{1}}\left(4\pi\sinh{}t\cdot
  \frac{\sinh(2t)}{\sqrt{\sinh^2(2t)-\sinh^2(2a)}}\right)dt\ .
\end{equation}
By Lemma \ref{lem:area compare}, for any
$a\geq{}a_{l}$ we have
\begin{equation*}
   \int_{a}^{\infty}\sinh{}t\cdot
   \left(\frac{\sinh(2t)}{\sqrt{\sinh^2(2t)-\sinh^2(2a)}}-1\right)dt
   <\cosh{}a-1\ ,
\end{equation*}
therefore for any $y_1\in(a,\infty)$ we have
\begin{equation*}
   4\pi\int_{a}^{y_1}\sinh{}t\cdot\left(\frac{\sinh(2t)}
   {\sqrt{\sinh^2(2t)-\sinh^2(2a)}}-1\right)dt
   <4\pi(\cosh{}a-1)\ ,
\end{equation*}
and then $\Area(\Sigma)<\Area(P_{+})+\Area(P_{-})$.
\end{proof}

\begin{claim}\label{claim:two}
There is no minimal annulus with the same boundary
as that of $\Sigma$ which has smaller area than that of $\Sigma$.
\end{claim}

\begin{proof}[\bf{Proof of Claim \ref{claim:two}}]
Recall that $a$ is an arbitrary constant chosen to be $\geq{}a_{l}$.
We need two notations:
\begin{itemize}
  \item $\Omega$ denotes the subregion of $\B^3$ bounded by
        $P(-x_{1})$ and $P(x_{1})$, and
  \item $\T_a$ denotes the simply connected
        subregion of $\B^3$ bounded  by $\CC_a$.
\end{itemize}

Assume that $\Sigma'$ is a least area annulus with the same boundary
as that of $\Sigma$, and $\Area(\Sigma')<\Area(\Sigma)$. Since
$\Sigma'$ is a least area annulus, it must be a minimal surface.
By \cite[Theorem 5]{MY1982(t)} and
\cite[Theorem 1]{MY1982(mz)}, $\Sigma'$ must be contained in
$\Omega$, otherwise we can use cutting and pasting technique
to get a minimal surface contained in $\Omega$ that has
smaller area. Furthermore, recall that
$\{\CC_\alpha\}_{\alpha\geq{}a_c}$ locally foliates
$\Omega\subset\B^3$, therefore $\Sigma'$ must be contained
in $\T_a\cap\Omega$ by the Maximum Principle. It's easy to
verify that the boundary of $\T_a\cap\Omega$ is given by
$\partial(\T_a\cap\Omega)=\Sigma\cup{}P_{+}\cup{}P_{-}$.

Now we claim that $\Sigma'$ is symmetric about any geodesic
plane that passes through the $u$-axis, i.e., $\Sigma'$ is
a surface of revolution. Otherwise, using the reflection
along the geodesic planes that pass through the $u$-axis,
we can find another annulus $\Sigma''$ with
$\partial\Sigma''=\partial\Sigma'$ such that either
$\Area(\Sigma'')<\Area(\Sigma')$ or $\Sigma''$ contains
\emph{folding curves} %(see \cite[pp. 418--419]{MY1982(t)})
so that we can find smaller area annulus by
the argument in \cite[pp. 418--419]{MY1982(t)}.
Similarly, $\Sigma'$ is symmetric about the $vw$-plane.

Now let $\sigma'=\Sigma'\cap\B_{+}^{2}$, then $\sigma'$
must satisfy the
equations \eqref{eq:differential equations of Pi} for some constant
$a'>0$, which may imply that $\Sigma'$ is a compact subdomain
of some catenoid $\CC_{a'}$ (see
Figure~\ref{fig:least area catenoid--lambda=1.2 and 0.14341}).
Obviously $\CC_{a'}\cap\CC_{a}=C_{+}\cup{}C_{-}$.

%====================================================================
\begin{figure}[htbp]
  \begin{center}
     \includegraphics[scale=0.9]{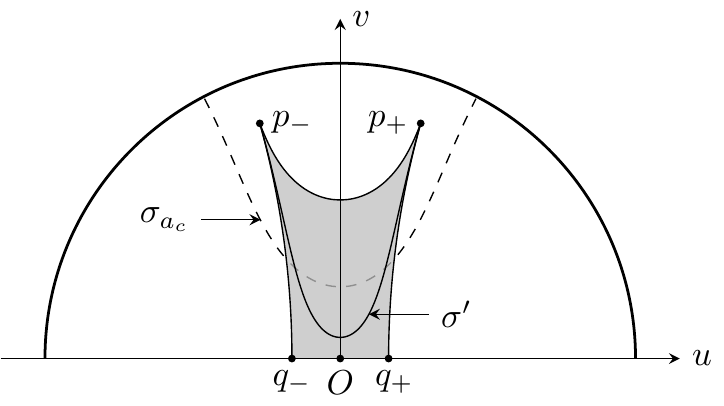}
  \end{center}
  \caption{The shaded region is equal to
  $(\T_a\cap\Omega)\cap\B^{2}_{+}$.
  If $\Sigma'\ne\Sigma$, then $\sigma'=\Sigma'\cap\B_{+}^{2}$ is the
  portion of some catenary $\sigma_{a'}$ with
  $a'<a_c$ and then $\Sigma'$ is unstable.}
  \label{fig:least area catenoid--lambda=1.2 and 0.14341}
\end{figure}
%====================================================================

Since $\Sigma'\subset\T_a\cap\Omega$, we have $a'\leq{}a$.
We claim that if $a'<a$, that is $\Sigma'$ is not the same
as $\Sigma$, then it must be unstable, which implies a contradiction.
In fact, if $a'<a$, since $\varrho(a')<\varrho(a)<\varrho(a_c)$,
we have $a'<a_c$ by Proposition~\ref{thm:Berard-Sa Earp}, which
implies that $\CC_{a'}$ is unstable. Besides, according to
Proposition \ref{thm:small maximally weakly stable domain},
the subdomain $\Sigma'$ of $\CC_{a'}$ is also unstable since
$\Sigma'$ contains the maximal weakly stable domain of $\CC_{a'}$,
so it couldn't be a least area
minimal surface unless $\Sigma'\equiv\Sigma$.

Therefore any compact annulus of the form
\eqref{eq:symmetric compact annulus}
is a least area minimal surface.
\end{proof}

Now let $S$ be any compact domain of $\CC_a$, then we
always can find a compact annulus $\Sigma$ of the form
\eqref{eq:symmetric compact annulus} such that $S\subset\Sigma$.
If $S$ is not a least area minimal surface, then we can
use the cutting and
pasting technique to show that $\Sigma$ is not a least
area minimal surface.
This is contradicted to the above argument.

Therefore if $a\geq{}a_l$, then $\CC_a$ is a least area minimal
surface in the sense of Meeks and Yau.
\end{proof}

\begin{remark}In the proof of \emph{Claim 2} in
Theorem \ref{thm:main theorem}, if $\Sigma'$ is an annulus type
minimal surface but it is not a least area
minimal surface, then it might not be a surface of revolution
(see \cite[p. 234]{Lop00}).
\end{remark}

\begin{corollary}
There exists a finite constant $a_{l}\approx{}1.10055$
such that for two disjoint circles $C_{1},C_{2}\subset{}S_{\infty}^{2}$,
if $d_{L}(C_{1},C_{2})\leq{}2\varrho(a_l)\approx{}0.72918$,
then there exist a least area spherical minimal catenoid $\CC$
which is asymptotic to $C_{1}\cup{}C_{2}$, where the function
$\varrho(a)$ is given by \eqref{eq:Gomes function I}.
\end{corollary}

%=================================================================
\section{Stability of Helicoids}\label{sec:stable helicoid}

In this section, we will prove Theorem \ref{thm:main theoremII}. One of
the key points in the proof is that two conjugate minimal surfaces
are either both stable or both unstable.

\subsection{Conjugate minimal surface}

Let $M^{3}(c)$ be the $3$-dimensional space form whose
sectional curvature is a constant $c$.

\begin{definition}[{\cite[pp. 699-700]{dCD83}}]
\label{def:conjugate minimal immersion}
Let $f:\Sigma\to{}M^{3}(c)$ be a minimal surface in isothermal
parametrs $(\sigma,t)$. Denote by
\begin{equation*}
   \I=E(d\sigma^2+dt^2)
   \quad\text{and}\quad
   \II=\beta_{11}d\sigma^2+2\beta_{12}d\sigma{}dt+
   \beta_{22}dt^2
\end{equation*}
the first and second fundamental forms of $f$, respectively.

Set $\psi=\beta_{11}-i\beta_{12}$ and define a family of
quadratic form depending on a parameter $\theta$,
$0\leq\theta\leq{}2\pi$, by
\begin{equation}
   \beta_{11}(\theta)=\Re(e^{i\theta}\psi)\ ,
   \quad
   \beta_{22}(\theta)=-\Re(e^{i\theta}\psi)\ ,
   \quad
   \beta_{12}(\theta)=\Im(e^{i\theta}\psi)\ .
\end{equation}
Then the following forms
\begin{equation*}
   \I_{\theta}=\I
   \quad\text{and}\quad
   \II_{\theta}=\beta_{11}(\theta)d\sigma^2+
   2\beta_{12}(\theta)d\sigma{}dt+
   \beta_{22}(\theta)dt^2
\end{equation*}
give rise to an isometry family
$f_{\theta}:\widetilde\Sigma\to{}M^{3}(c)$ of minimal
immersions, where $\widetilde\Sigma$ is the universal covering
of $\Sigma$. The immersion $f_{\pi/2}$ is called the
\emph{conjugate immersion} to $f_0=f$.
\end{definition}

The following result is obvious, but it's crucial to prove
Theorem \ref{thm:main theorem}.

\begin{lemma}Let $f:\Sigma\to{}M^{3}(c)$ be an immersed
minimal surface, and let
$f_{\pi/2}:\widetilde\Sigma\to{}M^{3}(c)$ be its conjugate
minimal surface, where $\widetilde\Sigma$ is the universal
covering of $\Sigma$. Then the minimal immersion $f$ is
globally stable if and only if its conjugate immersion
$f_{\pi/2}$ is globally stable.
\end{lemma}

\begin{proof}Let $\widetilde{f}$ be the universal lifting
of $f$, then
$\widetilde{f}:\widetilde\Sigma\to{}M^{3}(c)$ is a minimal
immersion. It's well known that the global stability of $f$
implies the global stability of $\widetilde{f}$.
Actually the minimal surfaces $\Sigma$ and $\widetilde\Sigma$
share the same Jacobi operator defined by
\eqref{eq:Jacobi operator}. If $\Sigma$ is globally stable,
there exists a positive Jacobi field on $\Sigma$ according to
Theorem \ref{thm:FCS80}, which implies that $\widetilde\Sigma$
is also globally stable, since the corresponding positive Jacobi
field on $\widetilde\Sigma$ is given by composing.

Next we claim that $\widetilde{f}$ and $f_{\pi/2}$ share the
same Jacobi operator. In fact, since the Laplacian depends only
on the first fundamental form, $\widetilde{f}$ and $f_{\pi/2}$
have the same Laplacian. Furthermore according to the definition
of the conjugate minimal immersion, $\widetilde{f}$ and $f_{\pi/2}$
have the same square norm of the second fundamental form, i.e.
$|A|^2=(\beta_{11}^2+2\beta_{12}^2+\beta_{22}^2)/E^2$, where we
used the notations in Definition
\ref{def:conjugate minimal immersion}.

By \eqref{eq:Jacobi operator} and Theorem \ref{thm:FCS80},
the proof of the lemma is complete.
\end{proof}

\subsection{Proof of Theorem \ref{thm:main theoremII}}

For hyperbolic and parabolic catenoids in $\HH^3$ defined in
$\S$\ref{subsub:hyperbolic catenoild-lorentz}
and $\S$\ref{subsub:parabolic catenoild-lorentz},
do Carmo and Dajczer proved that they are
globally stable (see \cite[Theorem (5.5)]{dCD83}). Furthermore,
Candel proved that the hyperbolic and parabolic catenoids are
least area minimal surfaces (see \cite[p. 3574]{Can07}).

The following result will be used for proving Theorem
\ref{thm:main theoremII}. The equation
\eqref{eq:relation between catenoids}
can be found in \cite[p. 34]{BSE09}.

\begin{lemma}[B{\'e}rard and Sa Earp]
\label{lem:relation between spherical catenoids}
The spherical catenoid $\Cscr_{1}(\atilde)$ defined
in $\S${}\ref{subsub:spherical catenoild-lorentz} is isometric to the
spherical catenoid $\CC_{a}$ defined in $\S${}\ref{subsec:catnoids},
where
\begin{equation}\label{eq:relation between catenoids}
   2\atilde=\cosh(2a)\ .
\end{equation}
\end{lemma}

\begin{proof}The spherical catenoid $\CC_{a}$ is
obtained by rotating the generating curve
$\sigma_{a}$ along the axis $\gamma_{0}$ in \eqref{eq:rotation axis}.
The distance between $\sigma_{a}$ and $\gamma_{0}$ is $a$.

The spherical catenoid $\Cscr_{1}(\atilde)$
defined in $\S${}\ref{subsub:spherical catenoild-lorentz}
can be obtained by rotating the generating
curve $C$ given by \eqref{eq:spherical catenoid-x1(s)}
and \eqref{eq:spherical catenoid-x3(s)-x4(s)} along the
geodesic $P^2\cap\HH^{3}$. The distance
between $C$ and $P^2\cap\HH^{3}$ is
$\sinh^{-1}\left(\sqrt{\atilde-1/2}\,\right)$.

Therefore the catenoid $\Cscr_{1}(\atilde)$ is
isometric to the catenoid $\CC_{a}$ if and only if
$a=\sinh^{-1}\left(\sqrt{\atilde-1/2}\,\right)$, which
implies \eqref{eq:relation between catenoids}.
\end{proof}

The following result can be found in \cite[Theorem (3.31)]{dCD83}.

\begin{theorem}[do Carmo-Dajczer]
\label{thm:conjugate minimal surfaces}
Let $f:\Cscr\to\HH^3$ be a minimal catenoid defined in
$\S${}\ref{subsec:catenoids in hyperboloid model}.
Its conjugate minimal surface is the geodesically-ruled minimal surface
$\Hcal_{\abar}$ given by \eqref{eq:helicoid in hyperboloid model} where
\begin{equation}\label{eq:relations for CMS}
   \begin{cases}
      \abar=\sqrt{(\atilde+1/2)/(\atilde-1/2)}\ ,
             &\text{if}\ \Cscr=\Cscr_{1}(\atilde)\
              \text{is spherical}\ ,\\
      \abar=\sqrt{(\atilde-1/2)/(\atilde+1/2)}\ ,
             &\text{if}\ \Cscr=\Cscr_{-1}(\atilde)\
              \text{is hyperbolic}\ ,\\
      \abar=1\ , &\text{if}\ \Cscr=\Cscr_{0}\
              \text{is parabolic}\ .
   \end{cases}
\end{equation}
\end{theorem}

Now we are able to prove the theorem.

\begin{maintheoremII}
For a family of minimal helicoids $\{\Hcal_{\abar}\}_{\abar\geq{}0}$
in the hyperbolic $3$-space given by
\eqref{eq:helicoid in hyperboloid model},
there exist a constant $\abar_c=\coth(a_c)\approx{}2.17968$ such that
the following statements are true:
  \begin{enumerate}
    \item $\Hcal_{\abar}$ is a globally stable minimal surface
          if $0\leq{}\abar\leq{}\abar_c$, and
    \item $\Hcal_{\abar}$ is an unstable minimal surface with Morse
          index one if $\abar>\abar_c$.
  \end{enumerate}
\end{maintheoremII}

\begin{proof}%[{\bf Proof of Theorem \ref{thm:main theorem}}]
When $\abar=0$, $\Hcal_{\abar}$ is a hyperbolic plane, so it is
globally stable.

According to Theorem~\ref{thm:conjugate minimal surfaces},
when $0<\abar<1$, $\Hcal_{\abar}$ is conjugate to
the hyperbolic catenoid $\Cscr_{-1}(\atilde)$, where
$\abar=\sqrt{(\atilde-1/2)/(\atilde+1/2)}$ by
\eqref{eq:relations for CMS}, and when $\abar=1$,
$\Hcal_{\abar}$ is conjugate to
the parabolic catenoid $\Cscr_{0}$ in $\HH^3$.
Therefore when $0<\abar\leq{}1$, the helicoid
$\Hcal_{\abar}$ is globally stable
by \cite[Theorem (5.5)]{dCD83}.

When $\abar>1$, $\Hcal_{\abar}$ is conjugate to the spherical catenoid
$\Cscr_{1}(\atilde)$ in $\HH^3$ by
Theorem~\ref{thm:conjugate minimal surfaces},
which is isometric to the spherical catenoid
$\CC_{a}$ in $\B^3$, where
$2\atilde=\cosh(2a)$
by \eqref{eq:relation between catenoids} and
$\abar=\sqrt{(\atilde+1/2)/(\atilde-1/2)}$ by
\eqref{eq:relations for CMS}.
Therefore $\Hcal_{\abar}$ is conjugate to
the spherical catenoid $\CC_{a}$ in $\B^3$, where
\begin{equation*}
   \abar=\coth(a)\ .
\end{equation*}
By Theorem \ref{thm:BSE theorem}, the catenoid $\CC_{a}$ is globally
stable if $a\geq{}a_{c}\approx{}0.49577$, therefore
$\Hcal_{\abar}$ is globally stable when
$1<\abar\leq{}\abar_c=\coth(a_{c})\approx{}2.17968$.

On the other hand, if $0<a<a_{c}$, then the
catenoid $\CC_{a}$ is unstable with Morse index one by
Theorem \ref{thm:BSE theorem}, therefore $\Hcal_{\abar}$
is unstable when $\abar>\abar_c$.

%-----------------------------------------------------------------------------------
\begin{figure}[htbp]
  \begin{center}
     \includegraphics[scale=0.15]{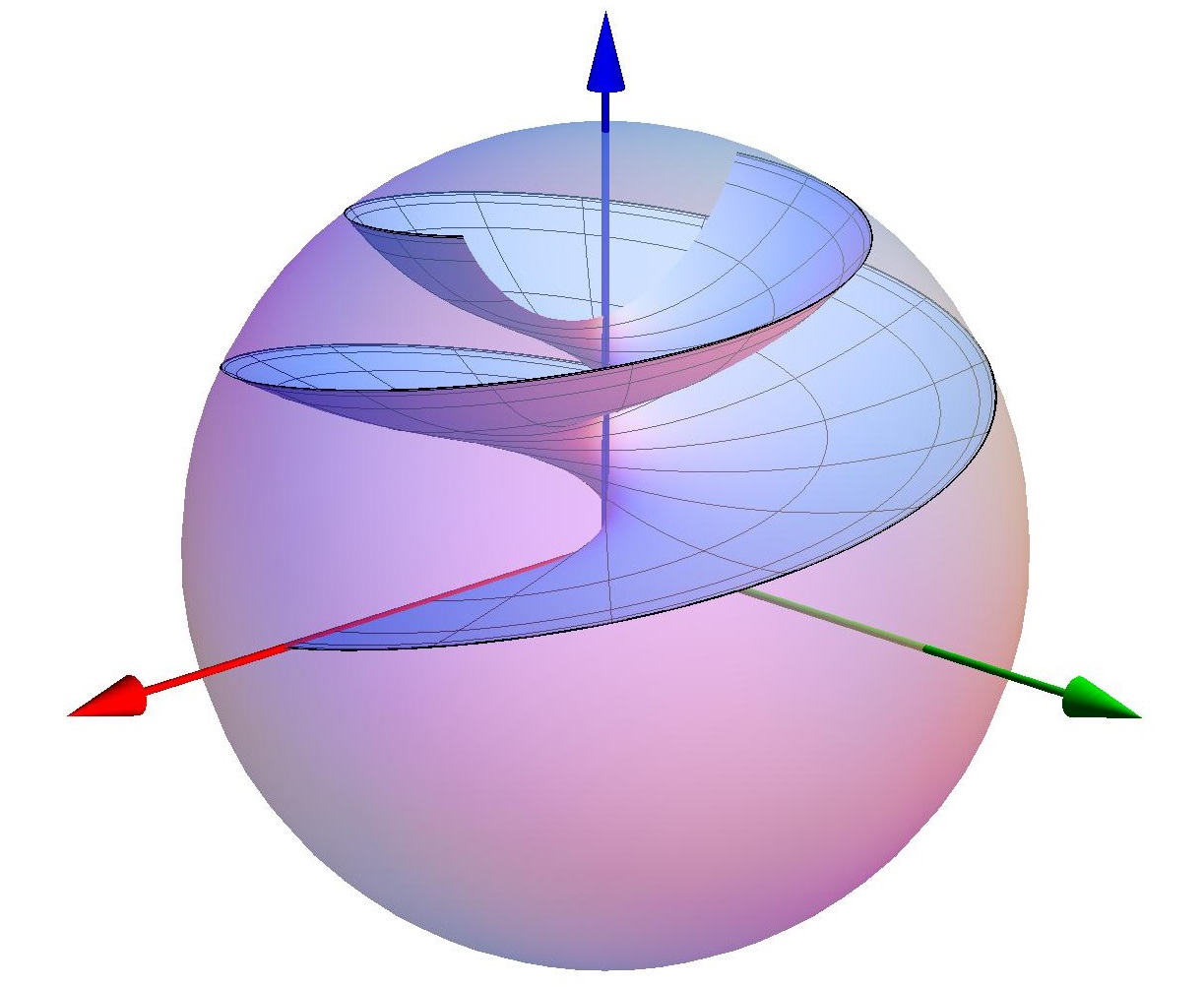}
  \end{center}
  \caption{This is a subdomain of the helicoid $\Hcal_5\subset\B^3$
  between two geodesics $v=0$ and $v=\pi/\sqrt{6}$ of
  $\Hcal_{5}$, which is isometric to the fundamental domain of the universal cover
  of the catenoid $\CC_{a}$ for $a=\coth^{-1}(5)$.}
  \label{fig:fundamental domain of catenoid}
\end{figure}
%-----------------------------------------------------------------------------------

At last we will show that the Morse index of the helicoid
$\Hcal_{\abar}$ for $\abar>\abar_c$ is infinite. Otherwise
according to Theorem~\ref{prop:FC85}, there is a compact
subdomain $K$ of $\Hcal_{\abar}$ such that
$\Hcal_{\abar}\setminus{}K$ is stable. But this is impossible,
since we always can find a (noncompact) subdomain $\mathcal{U}$
of $\Hcal_{\abar}$ as follows: Consider the parametric equation
\eqref{eq:helicoid--ball model} of $\Hcal_{\abar}$ in the Poincar{\'e}
ball model $\B^3$ of the hyperbolic $3$-space, and write
$\mathcal{U}$ by
\begin{equation*}
   \mathcal{U}=\left\{(x,y,z)\in\Hcal_{\abar}\ |\
   u\in\R\ \text{and}\
   v_{0}\leq{}v\leq{}v_{0}+2\pi\sinh(\coth^{-1}(\abar))\right\}\ ,
\end{equation*}
where $v_{0}>0$ is sufficiently large such that $K$ is underneath
the geodesic of $\B^3$ which passes through
the point $(0,0,\tanh(v_{0}/2))\in\B^3$ and is contained in $\Hcal_{\abar}$
(see Figure~\ref{fig:fundamental domain of catenoid}).
It's easy to verify that
\begin{itemize}
  \item $\mathcal{U}$ is disjoint from $K$, i.e.
        $\mathcal{U}\subset\Hcal_{\abar}\setminus{}K$, and
  \item $\mathcal{U}$ is isometric to the fundamental domain
        of the universal cover of the catenoid $\CC_{a}$, where
        $a=\coth^{-1}(\abar)<a_{c}$.
\end{itemize}
Therefore $\mathcal{U}$ is unstable, and so is
$\Hcal_{\abar}\setminus{}K$.
\end{proof}

%===============================================================================
\section{Technical Lemmas}\label{sec:technical lemmas}

\subsection{Lemmas for $\S$\ref{sec:stable catenoid}}

In this part, we prove Lemma \ref{lem:derivative of d0} and
Lemma \ref{lem:second derivative of d[lambda]}, which will be
applied to prove Theorem~\ref{thm:BSE theorem}.

\begin{lemma}\label{lem:derivative of d0}
Let $\phi(a,t)=\sqrt{5}\,\cosh(a+t)-\cosh(3a+t)$,
then $\phi(a,t)\leq{}0$ for
$(a,t)\in{}[A_3,\infty)\times[0,\infty)$,
where the constant $A_{3}$ is defined by
\begin{equation}\label{eq:constant Lambda-3}
   A_3=\cosh^{-1}\left(\frac{\sqrt{3+\sqrt{5}}}{2}\right)
            \approx 0.530638\ .
\end{equation}
\end{lemma}

\begin{proof}It's easy to verify that $\phi(a,t)\leq{}0$ is
equivalent to
\begin{equation*}
   \cosh(3a)-\sqrt{5}\cosh{}a+\tanh{}t\cdot
   (\sinh(3a)-\sqrt{5}\sinh{}a)\geq{}0\ .
\end{equation*}
Since $\tanh{}t\geq{}0$ for $t\geq{}0$ and
$\sinh(3a)-\sqrt{5}\sinh{}a\geq{}0$ for $a\geq{}0$,
we need solve the inequality
$\cosh(3a)-\sqrt{5}\cosh{}a\geq{}0$.

Let $A_3$ be the solution of the equation
$0=\sqrt{5}\,\cosh{}a-\cosh(3a)=
(\sqrt{5}-(4\cosh^{2}a-3))\cosh{}a$,
then $\phi(a,t)\leq{}0$
if $a\geq{}A_3$ and $t\geq{}0$.
\end{proof}

\begin{lemma}\label{lem:second derivative of d[lambda]}
Let $\psi(a,t)$ be the function given by
\begin{equation}\label{eq:psi(t,l)}
\begin{aligned}
   \psi(a,t)=
        &\,76\sinh(2a)-22\sinh(2t)+29\sinh(4a+2t)\\
        &\,+\sinh(8a+2t)-26\sinh(6a+4t)-6\sinh(10a+4t)\\
        &\,-25\sinh(8a+6t)+\sinh(12a+6t)\ .
\end{aligned}
\end{equation}
Then $\psi(a,t)<0$ for all $(a,t)\in{}[0,A_4]\times[0,\infty)$,
where the constant
\begin{equation}\label{eq:constant Lambda-4}
   A_4=\frac{1}{4}\cosh^{-1}\left(\frac{35+\sqrt{1241}}{8}\right)
            \approx 0.715548
\end{equation}
is the solution of the equation
$4\cosh^2(4a)-35\cosh(4a)-1=0$.
\end{lemma}

\begin{proof}Expand each term in $\psi(a,t)$ with the form
$\sinh(ma+nt)$ , then we may write
$\psi(a,t)=\psi_1(a,t)+\psi_2(a,t)$, where
\begin{equation*}
\begin{aligned}
   \psi_1(a,t)=
        &\,-22\sinh(2t)+29\sinh(2t)\cosh(4a)+\sinh(2t)\cosh(8a)\\
        &\,-26\sinh(4t)\cosh(6a)-6\sinh(4t)\cosh(10a)\\
        &\,-25\sinh(6t)\cosh(8a)+\sinh(6t)\cosh(12a))\ ,
\end{aligned}
\end{equation*}
and
\begin{equation*}
\begin{aligned}
   \psi_2(a,t)=
        &\,76\sinh(2a)+29\cosh(2t)\sinh(4a)+\cosh(2t)\sinh(8a)\\
        &\,-26\cosh(4t)\sinh(6a)-6\cosh(4t)\sinh(10a)\\
        &\,-25\cosh(6t)\sinh(8a)+\cosh(6t)\sinh(12a)\ .
\end{aligned}
\end{equation*}

\noindent\textbf{Claim}: \emph{$\psi_1(a,t)\leq{}0$ and
$\psi_2(a,t)\leq{}0$ for $(a,t)\in[0,A_4]\times[0,\infty)$}.

\begin{proof}[\bf{Proof of Claim}]
First of all, we will show that $\psi_1(a,\cdot)\leq{}0$
for $a\in[0,A_4]$. Since $\cosh(2t)\geq{}1$ for
any $t\in[0,\infty)$, we have the estimate
\begin{align*}
   \psi_1(a,t)
       =&\,-\sinh(2t)(22-29\cosh(4a)-\cosh(8a)\\
        &\qquad\qquad +52\cosh(2t)\cosh(6a)
                      +12\cosh(2t)\cosh(10a))\\
        &\,-\sinh(6t)(25\cosh(8a)-\cosh(12a))\\
    \leq&\,-\sinh(2t)(22-29\cosh(4a)-\cosh(8a)\\
        &\qquad\qquad +52\cosh(6a)+12\cosh(10a))\\
        &\,-\sinh(6t)(25\cosh(8a)-\cosh(12a))\ .
\end{align*}
Since $52\cosh(6a)-29\cosh(4a)>0$ and
$12\cosh(10a)-\cosh(8a)>0$ for $0\leq{}a<\infty$ and
$25\cosh(8a)-\cosh(12a)>0$ for $0\leq{}a\leq{}A_4$,
we have $\psi_1(a,\cdot)<0$ for $0\leq{}a\leq{}A_4$.

Secondly, for any $a\geq{}0$, we apply the inequality
\begin{equation*}
   \sinh((m+n)a)\geq\sinh(ma)+\sinh(na)\ ,
\end{equation*}
where $m,n$ are positive integers, to get the following
inequalities
\begin{itemize}
   \item $\sinh(6a)\geq\sinh(4a)+\sinh(2a)$,
   \item $\sinh(8a)\geq{}4\sinh(2a)$, and
   \item $\displaystyle\sinh(10a)\geq
          \begin{cases}
             \sinh(8a)+\sinh(2a)\\
             \sinh(4a)+3\sinh(2a)\\
             5\sinh(2a)
          \end{cases}$\ ,
\end{itemize}
which can imply the estimate
\begin{align*}
   \psi_2(a,t)\leq
        &\,-46\sinh(2a)(\cosh(4t)-1)-30\sinh(2a)(\cosh(6t)-1)\\
        &\,-29\sinh(4a)(\cosh(4t)-\cosh(2t))\\
        &\,-\sinh(8a)(\cosh(4t)-\cosh(2t))\\
        &\,-\cosh(6t)\left(\frac{35}{2}
            \sinh(8a)-\sinh(12a)\right)\ .
\end{align*}
As $\frac{35}{2}\sinh(8a)-\sinh(12a)=
\sinh(4a)(1+35\cosh(4a)-4\cosh^2(4a))\geq{}0$ if
$0\leq{}a\leq{}A_4$ and the fact
$\cosh(6t)\geq\cosh(4t)\geq\cosh(2t)\geq{}1$ for $0\leq{}t<\infty$,
we have $\psi_2(a,t)\leq{}0$ for
$(a,t)\in[0,A_4]\times[0,\infty)$.
\end{proof}

Therefore $\psi(a,t)<0$ for $(a,t)\in[0,A_4]\times[0,\infty)$.
\end{proof}

\subsection{Lemmas for $\S$\ref{sec:least area catenoid}}

In this subsection, we will prove Proposition~\ref{thm:Berard-Sa Earp}
and Proposition~\ref{thm:small maximally weakly stable domain}, which
will be used to prove Theorem~\ref{thm:main theorem}.
At first, we need the following lemma.

\begin{lemma}\label{lem:negative second order derivative of d(a,t)}
Let $R_3$ and $R_4$ be the regions defined by
\begin{align*}
  R_3&=\{(a,t)\in\R^2\ |\ t\geq{}a\geq{}A_3\}\ ,\\
  R_4&=\{(a,t)\in\R^2\ |\ 0<a\leq{}A_4\ \text{and}\ t\geq{}a\}\ ,
\end{align*}
where $A_3$ and $A_4$ are the constants defined in
\eqref{eq:constant Lambda-3} and \eqref{eq:constant Lambda-4}.
Then we have
\begin{equation*}
  \ppl{}{a}\,\rho(a,t)<0
  \quad\text{for}\ (a,t)\in{}R_3\ ,
\end{equation*}
and
\begin{equation*}
  \ppz{}{a}\,\rho(a,t)<0
  \quad\text{for}\ (a,t)\in{}R_4\ ,
\end{equation*}
where $\rho(a,t)$ is defined in \eqref{eq:catenary(a,t)}.
\end{lemma}

\begin{proof}Using the substitution $\tau\mapsto\tau+a$ in
\eqref{eq:catenary(a,t)}, we have
\begin{equation*}
  \rho(a,t)=\int_{0}^{t-a}\frac{\sinh(2a)}{\cosh(a+\tau)}
  \frac{d\tau}{\sqrt{\sinh^2(2a+2\tau)-\sinh^2(2a)}}\ ,
  \quad{}t\geq{}a\ .
\end{equation*}
Direct computation shows
\begin{align*}
  \ppl{}{a}\,\rho(a,t)=
     &\,\int_{0}^{t-a}\frac{\sinh(a+\tau)
      (5\cosh^2(a+\tau)-\cosh^{2}(3a+\tau))}
      {\cosh^2(a+\tau)\sqrt{\sinh(2\tau)\sinh^{3}(4a+2\tau)}}\,
      d\tau\\
     &-\frac{\sinh(2a)}
      {\cosh{}t\cdot\sqrt{\sinh^{2}(2t)-\sinh^{2}(2a)}}
\end{align*}
and
\begin{align*}
   \ppz{}{a}\,\rho(a,t)=
       &\,\int_{0}^{t-a}\frac{\psi(a,\tau)}
          {16\cosh^3(a+\tau)\sqrt{\sinh(2\tau)\sinh^{5}(4a+2\tau)}}
          \,d\tau\\
       &\,-\frac{\sinh{}t\cdot{}w(a,t)}
          {\sqrt{2}\,\cosh^{2}t\cdot(\cosh(4t)-\cosh(4a))^{3/2}}\ ,
\end{align*}
where $\psi(a,\tau)$ is given by \eqref{eq:psi(t,l)} and $w(a,t)$ is defined by
\begin{align*}
  w(a,t)=&\,-5\sinh(2a)+\sinh(6a)-7\sinh(2a-4t)\\
         &\,-12\sinh(2a-2t)+4\sinh(2a+2t)+\sinh(2a+4t)\ .
\end{align*}
Recall that $t\geq{}a>0$, we have $w(a,t)\geq\sinh(6a)>0$, together with
the arguments in the proofs of Lemma \ref{lem:derivative of d0}
and Lemma \ref{lem:second derivative of d[lambda]}, the proof
of the lemma is complete.
\end{proof}

\begin{proposition}[{\cite[Proposition 4.8 and Lemma 4.9]{BSE10}}]
\label{thm:Berard-Sa Earp}
Let $\sigma_a\subset\B_{+}^2$ be the catenary given by
\eqref{eq: catenary equation}. For $0<a_1<a_2$,
the catenaries $\sigma_{a_1}$ and $\sigma_{a_2}$
intersect at most at two symmetric points and they do so if and
only if $\varrho(a_1)<\varrho(a_2)$. Furthermore we have the following
results:
\begin{enumerate}
   \item For $a_1,a_2\in(0,a_c)$, the catenaries
         $\sigma_{a_1}$ and $\sigma_{a_2}$ intersect exactly
         at two symmetric points {\rm(}see
         Figure~\ref{fig:unstable catenaries}{\rm)}.
   \item All catenaries in the family $\{\sigma_a \ |\ a\geq{}a_c\}$ foliate
         the subdomain of $\B^2_{+}$ that is bounded by the
         catenary $\sigma_{a_c}$ and the arc of
         $\partial_{\infty}\B^2_{+}$ between
         the asymptotic boundary points of $\sigma_{a_c}$
         {\rm(}see Figure~\ref{fig:stable catenaries}{\rm)}.
\end{enumerate}
\end{proposition}

\begin{proof}For $0<a_1<a_2$, we define a function
\begin{equation*}
   \delta(t)=\rho(a_{2},t)-\rho(a_{1},t)\ ,
   \quad{}t\geq{}a_{2}>a_{1}\ ,
\end{equation*}
where $\rho$ is defined by \eqref{eq:catenary(a,t)}.

We claim that $\delta(t)$ is an increasing function
on $[a_{2},\infty)$. Actually the derivative of $\delta(t)$ is
given by the following expression
\begin{equation*}
   \delta'(t)=\frac{1}{\cosh{}t}
   \left\{\frac{1}{\sqrt{\left(\dfrac{\sinh(2t)}{\sinh(2a_{2})}\right)^2-1}}
   -\frac{1}{\sqrt{\left(\dfrac{\sinh(2t)}{\sinh(2a_{1})}\right)^2-1}}\right\}>0\ ,
\end{equation*}
on $[a_{2},\infty)$.

Since $\delta(0)=-\rho(a_{2},a_{1})<0$ and
$\delta(\infty)=\varrho(a_{2})-\varrho(a_{1})$, we have the conclusion that
the catenaries $\sigma_{a_1}$ and $\sigma_{a_2}$
intersect at most at two symmetric points and they do so if and
only if $\varrho(a_1)<\varrho(a_2)$.

Statement (1) is true since $\varrho(a)$ is increasing on $(0,a_{c})$,
and Statement (2) is true since $\varrho(a)$ is decreasing
on $[a_{c},\infty)$
\end{proof}

%=================================================================
\begin{figure}[htbp]
\begin{center}
  \begin{minipage}[tb]{0.5\textwidth}
  \centering
  \includegraphics[scale=1.15]{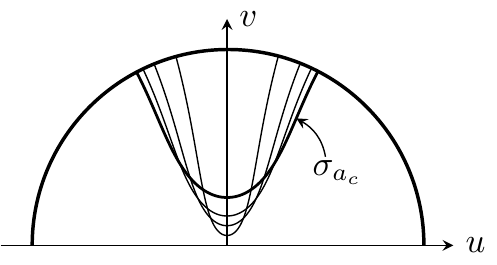}
  \caption{The family of catenaries for unstable catenoids.}
  \label{fig:unstable catenaries}
  \end{minipage}%
  \begin{minipage}[tb]{0.5\textwidth}
  \centering
  \includegraphics[scale=1.15]{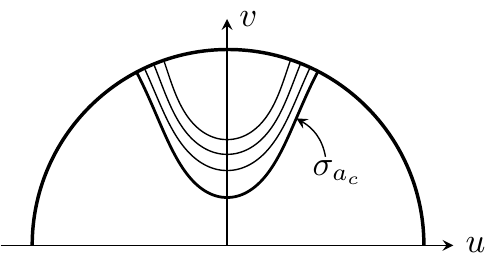}
  \caption{The family of catenaries for stable catenoids.}
  \label{fig:stable catenaries}
  \end{minipage}
\end{center}
\end{figure}
%==================================================================

An \emph{envelope} of a family of curves in the hyperbolic
plane is a curve that is tangent to each member of the family at
some point. A point on the envelope can be considered
as the intersection of two ``adjacent'' curves, meaning the limit of
intersections of nearby curves.

According to Proposition~\ref{thm:Berard-Sa Earp}, the catenaries
$\sigma_{a_1}$ and $\sigma_{a_2}$ intersect exactly at two symmetric points
if $0<a_1<a_2<a_c$. In order to prove
Theorem~\ref{thm:main theorem}, we should show that the family of
catenaries $\{\sigma_{a}\}_{0<a<a_{c}}$ forms an envelope which
is outside the region of $\B^{2}_{+}$ foliated by the family
of catenaries $\{\sigma_a\}_{a\geq{}a_c}$, since each point in
the envelope is corresponding to the boundary of the maximal
weakly stable domain of some catenoid $\CC_{a}$ for $0<a<a_{c}$.
%More precisely, we have the following result.

Recall that the the numbers $\pm{}z(a)$ denote the
only zeros of the variation field $\xi(a,s)$ on the catenoid $\CC_{a}$
defined by \eqref{eq:variation fiels I} or \eqref{eq:variation fiels II}
for each $0<a<a_{c}$.

\begin{proposition}[{\cite[Proposition 4.8 (2)]{BSE10}}]
\label{thm:small maximally weakly stable domain}
The family of catenaries $\{\sigma_{a}\ |\ 0<a<a_{c}\}$
defined in $\S$\ref{sec:catenary} has an envelope,
and the points at which $\sigma_{a}$ touches the envelope correspond to the
maximal weakly stable domain $\CC(z(a))$ defined by \eqref{eq:portion of Pi}.
Furthermore for any constant $a\in(0,a_c)$ we have
\begin{equation}\label{eq:maximal weakly stable region}
  \CC(z(a))\cap\Bigg(\bigcup_{\alpha\geq{}a_c}
  \CC_{\alpha}\Bigg)=\emptyset\ .
\end{equation}
\end{proposition}

\setcounter{claim}{0} %setcounter{claim}{0}

%==================================================================
\begin{figure}[htbp]
  \begin{center}
     \includegraphics[scale=0.9]{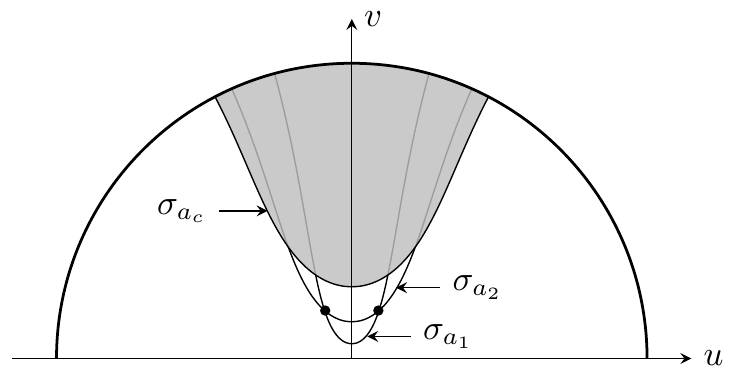}
  \end{center}
  \caption{$\mathcal{D}$ is the shadow region. In this figure,
  $a_1=0.1$ and $a_2=0.25$. We can see that
  $\sigma_{a_1}\cap\sigma_{a_2}$ is disjoint from
  $\mathcal{D}$.}
  \label{fig:catenary_lambda_shadow}
\end{figure}
%===================================================================

\begin{proof}For each $a\in(0,a_{c})$, the catenoid $\CC_{a}$ and its
maximal weakly stable domain $\CC(z(a))$ are surfaces of revolution, so
we just need to consider the $2$-dimensional case. Let
\begin{equation*}
   \mathcal{D}=\bigcup_{a\geq{}a_c}\sigma_{a}
\end{equation*}
be the subregion of $\B_{+}^2$. We will prove the proposition by three claims.

\begin{claim}\label{claim:01}
$\sigma_{a_1}\cap\sigma_{a_2}$ is disjoint from
$\mathcal{D}$ for $0<a_1<a_2<a_c$.
\end{claim}

\begin{proof}[{\bf{Proof of Claim~\ref{claim:01}}}]
See Figure~\ref{fig:catenary_lambda_shadow}.
Actually if $\sigma_{a_1}\cap\sigma_{a_2}\subset{}\mathcal{D}$,
since $\mathcal{D}$ is foliated
by the catenaries $\{\sigma_\alpha\}_{\alpha\geq{}a_c}$, there
exists $a_3>a_c$ such that $\sigma_{a_1}$, $\sigma_{a_2}$
and $\sigma_{a_3}$ intersect at the same points. Let $t_0$ be the
$y$-coordinate of the intersection points (recall that we equip
$\B^{2}_{+}$ with the warped product metric
\eqref{eq:warped product metric}), then $t_0>a_3$ (see the proof
of Lemma 4.9 in \cite{BSE10}). Consider the function
\begin{equation*}
  \varphi(\alpha)=\rho(\alpha,t_0)\ ,
  \quad{}\alpha\in[a_1,a_3]\ ,
\end{equation*}
where $\rho(a,t)$ is defined by \eqref{eq:catenary(a,t)}.

By our assumption, $\varphi(a_1)=\varphi(a_2)=\varphi(a_3)$. By
L'H{\^o}pital's rule,
there exist $a_4\in(a_1,a_2)$ and $a_5\in(a_2,a_3)$ such that
$\varphi'(a_4)=\varphi'(a_5)=0$. By
Lemma~\ref{lem:negative second order derivative of d(a,t)},
$a_5<A_3$ and then $a_5<A_4$. Applying L'H{\^o}pital's rule again,
there exists $a_6\in(a_4,a_5)\subset(a_1,a_3)\cap(0,A_4)$ such that
$\varphi''(a_6)=0$. This is impossible according to
Lemma~\ref{lem:negative second order derivative of d(a,t)}.
Therefore $\sigma_{a_1}\cap\sigma_{a_2}$ is disjoint from
$\mathcal{D}$ for $0<a_1<a_2<a_c$.
\end{proof}

\begin{claim}\label{claim:02}
For any $a\in(0,a_{c})$, there exist two symmetric points
$p_{a}^{\pm}\subset\sigma_{a}$ disjoint from $\mathcal{D}$ such that
\begin{equation}\label{eq:endpoints of generating curve of MWSR}
   \lim\limits_{\alpha\to{}a}(\sigma_a\cap\sigma_\alpha)=\{p_{a}^{\pm}\}\ .
\end{equation}
\end{claim}

\begin{proof}[{\bf{Proof of Claim~\ref{claim:02}}}]
For $0<a_1<a_2<a_3<a_c$, let $y_{ij}$ denote the $y$-coordinate of
the intersection of $\sigma_{a_i}$ and $\sigma_{a_j}$
($1\leq{}i<j\leq{}3$), then we have $y_{12}<y_{13}<y_{23}$
(see Figure~\ref{fig:catenary_intersection}).

%==================================================================
\begin{figure}[htbp]
  \begin{center}
     \includegraphics[scale=0.9]{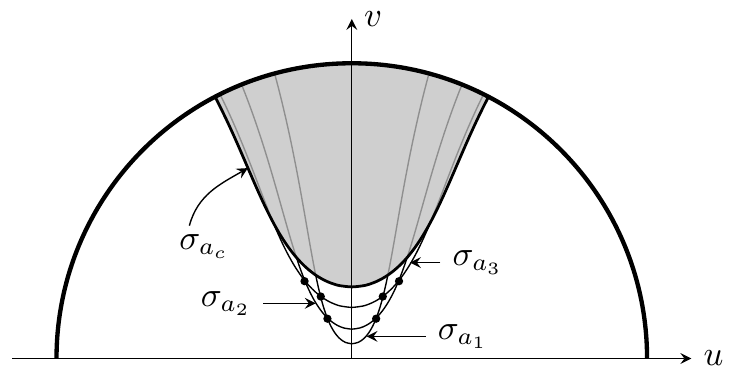}
  \end{center}
  \caption{For $0<a_i<a_j<a_c$, let $y_{ij}$ denote the $y$-coordinate
  of the (symmetric) points $\sigma_{a_i}\cap\sigma_{a_j}$. In this figure,
  $a_1=0.1$, $a_2=0.2$ and $a_3=0.35$, and one can see that
  $y_{12}<y_{13}<y_{23}$.}
  \label{fig:catenary_intersection}
\end{figure}
%===================================================================

Otherwise, we must have
$y_{23}<y_{13}<y_{12}$, this may imply there exists $a_4\in(a_3,a_c)$
such that $\sigma_{a_1}$, $\sigma_{a_2}$ and $\sigma_{a_4}$ intersect
at the same points. But this is impossible according to the
similar argument as in Claim \ref{claim:01}.

Therefore as $\alpha\to{}a^{+}$, the $y$-coordinates of
$\sigma_{a}\cap\sigma_{\alpha}$ are decreasing, whereas
as $\alpha\to{}a^{-}$, the $y$-coordinates of
$\sigma_{a}\cap\sigma_{\alpha}$ are increasing, which can imply that
$\lim\limits_{\alpha\to{}a}(\sigma_a\cap\sigma_\alpha)$
exists, and the points $\{p_{a}^{\pm}\}$ are disjoint from $\mathcal{D}$
because of the statement in Claim~\ref{claim:01}.
\end{proof}

By Claim~\ref{claim:01} and Claim~\ref{claim:02}, we actually have proved
that the catenaries $\{\sigma_{a}\}_{0<a<a_{c}}$ has an envelope curve,
denoted by
\begin{equation*}
   \Gamma=\{p_{a}^{\pm}\ |\ 0<a<a_{c}\}\cup\{(0,0)\}\ ,
\end{equation*}
which is disjoint from the region
$\mathcal{D}$ (see Figure~\ref{fig:catenary_envelope_curve}).

Recall that each catenary $\sigma_{a}$ can be parametrized by arc length as follows:
\begin{equation*}
   s\mapsto(x(a,s),y(a,s))
\end{equation*}
for $-\infty<s<\infty$, where $x(\cdot,\cdot)$ and $y(\cdot,\cdot)$ are defined by
the equations \eqref{eq:x(a,s)} and \eqref{eq:y(a,s)} respectively.
For each $a\in(0,a_{c})$, suppose that
\begin{equation*}
   p_{a}^{\pm}=(\pm{}x(a,s_a),y(a,s_a))\ .
\end{equation*}

\begin{claim}\label{claim:03}
For any $a\in(0,a_{c})$, we have $z(a)=s_{a}$, where
$\pm{}z(a)$ are the only zeros of the variation field
$\xi(a,\cdot)$ given by \eqref{eq:variation fiels I}
and \eqref{eq:variation fiels II}.
\end{claim}

\begin{proof}[{\bf{Proof of Claim~\ref{claim:03}}}]
By the property of the envelope curve $\Gamma$, it
is tangent to each catenary $\sigma_a$ at the points
$p_{a}^{\pm}$ for $0<a<a_{c}$. Thus $\sigma_a$ and $\Gamma$ have
the same tangent line at either $p_{a}^{+}$ or $p_{a}^{-}$.

The tangent vector to the catenary $\sigma_a$ at $p_{a}^{+}$
is $(x_{s}(a,s_a),y_{s}(a,s_a))$, and the
tangent vector to the envelope $\Gamma$ at $p_{a}^{+}$
is $(x_{a}(a,s_a),y_{a}(a,s_a))$. These two vectors
must be proportional, therefore we have
\begin{equation*}
   x_{a}(a,s_a)y_{s}(a,s_a)-x_{s}(a,s_a)y_{a}(a,s_a)=0
\end{equation*}
for all $0<a<a_{c}$. Now by \eqref{eq:variation fiels I},
the variation field $\xi(a,s)$ on the catenoid $\CC_{a}$ has two symmetric
zeros at $s_{a}$ and $-s_{a}$. On the other hand, according to
Lemma \ref{lem:BSE-stability-half-catenoid}, these are the only zeros of
$\xi(a,s)$ for $s\in(-\infty,\infty)$. Thus we have $z(a)=s_a$ for all
$0<a<a_c$.
\end{proof}

Now according to Claim \ref{claim:03}, we have
\eqref{eq:maximal weakly stable region}, and then the proof of the whole proposition
is complete.
\end{proof}

%==================================================================
\begin{figure}[htbp]
  \begin{center}
     \includegraphics[scale=0.9]{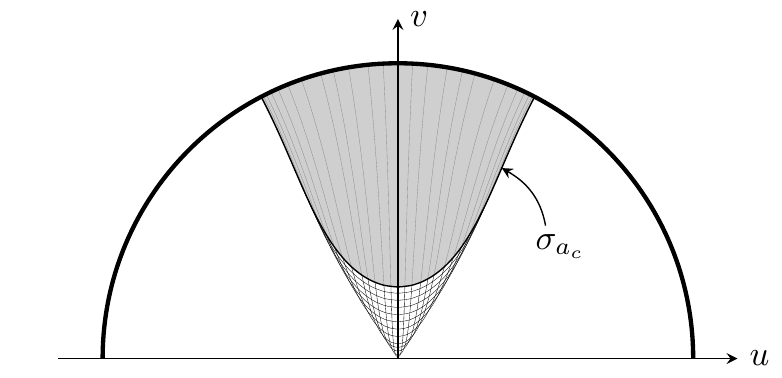}
  \end{center}
  \caption{$\mathcal{D}$ is the shadow region. The family of catenaries
  $\sigma_{a}$ for $0<a<a_{c}$ forms an envelope curve, which is disjoint
  from $\mathcal{D}$.}
  \label{fig:catenary_envelope_curve}
\end{figure}
%===================================================================

%--------------------------------------------------------------------
%\bibliographystyle{plain}
\bibliographystyle{amsalpha}
\bibliography{ref_stability}

\providecommand{\bysame}{\leavevmode\hbox to3em{\hrulefill}\thinspace}
\providecommand{\MR}{\relax\ifhmode\unskip\space\fi MR }
\providecommand{\MRhref}[2]{%
  \href{http://www.ams.org/mathscinet-getitem?mr=#1}{#2}
}
\providecommand{\href}[2]{#2}
\begin{thebibliography}{dCGT86}

\bibitem[BdC80]{BdC80b}
Jo{\~a}o~Lucas Barbosa and Manfredo do~Carmo, \emph{Stability of minimal
  surfaces in spaces of constant curvature}, Bol. Soc. Brasil. Mat. \textbf{11}
  (1980), no.~1, 1--10.

\bibitem[BdCH09]{BCH09}
Allen Back, Manfredo do~Carmo, and Wu-Yi Hsiang, \emph{On some fundamental
  equations of equivariant {R}iemannian geometry}, Tamkang J. Math. \textbf{40}
  (2009), no.~4, 343--376.

\bibitem[Bea95]{Bea95}
Alan~F. Beardon, \emph{The geometry of discrete groups}, Graduate Texts in
  Mathematics, vol.~91, Springer-Verlag, New York, 1995, Corrected reprint of
  the 1983 original.

\bibitem[BP92]{BP92}
Riccardo Benedetti and Carlo Petronio, \emph{Lectures on hyperbolic geometry},
  Universitext, Springer-Verlag, Berlin, 1992.

\bibitem[BSE09]{BSE09}
Pierre B{\'e}rard and Ricardo Sa~Earp, \emph{Lindel{\"o}f's theorem for
  catenoids revisited}, arXiv:0907.4294v1 (2009).

\bibitem[BSE10]{BSE10}
\bysame, \emph{Lindel\"of's theorem for hyperbolic catenoids}, Proc. Amer.
  Math. Soc. \textbf{138} (2010), no.~10, 3657--3669.

\bibitem[Can07]{Can07}
Alberto Candel, \emph{Eigenvalue estimates for minimal surfaces in hyperbolic
  space}, Trans. Amer. Math. Soc. \textbf{359} (2007), no.~8, 3567--3575
  (electronic).

\bibitem[CG06]{CG06}
Danny Calegari and David Gabai, \emph{Shrinkwrapping and the taming of
  hyperbolic 3-manifolds}, J. Amer. Math. Soc. \textbf{19} (2006), no.~2,
  385--446.

\bibitem[CM11]{CM11}
Tobias~Holck Colding and William~P. Minicozzi, II, \emph{A course in minimal
  surfaces}, Graduate Studies in Mathematics, vol. 121, American Mathematical
  Society, Providence, RI, 2011.

\bibitem[dCD83]{dCD83}
Manfredo do~Carmo and Marcos Dajczer, \emph{Rotation hypersurfaces in spaces of
  constant curvature}, Trans. Amer. Math. Soc. \textbf{277} (1983), no.~2,
  685--709.

\bibitem[dCGT86]{dCGT86}
Manfredo do~Carmo, Jonas de~Miranda Gomes, and Gudlaugur Thorbergsson,
  \emph{The influence of the boundary behaviour on hypersurfaces with constant
  mean curvature in {$H^{n+1}$}}, Comment. Math. Helv. \textbf{61} (1986),
  no.~3, 429--441.

\bibitem[dCP79]{dCP79}
M.~do~Carmo and C.~K. Peng, \emph{Stable complete minimal surfaces in {${\bf
  R}^{3}$} are planes}, Bull. Amer. Math. Soc. (N.S.) \textbf{1} (1979), no.~6,
  903--906.

\bibitem[dOS98]{dOS98}
Geraldo de~Oliveira and Marc Soret, \emph{Complete minimal surfaces in
  hyperbolic space}, Math. Ann. \textbf{311} (1998), no.~3, 397--419.

\bibitem[FC85]{Fis85}
D.~Fischer-Colbrie, \emph{On complete minimal surfaces with finite {M}orse
  index in three-manifolds}, Invent. Math. \textbf{82} (1985), no.~1, 121--132.

\bibitem[FCS80]{FCS80}
Doris Fischer-Colbrie and Richard Schoen, \emph{The structure of complete
  stable minimal surfaces in {$3$}-manifolds of nonnegative scalar curvature},
  Comm. Pure Appl. Math. \textbf{33} (1980), no.~2, 199--211.

\bibitem[FT91]{FT91}
Anatoli{\u{\i}}~T. Fomenko and Alexey~A. Tuzhilin, \emph{Elements of the
  geometry and topology of minimal surfaces in three-dimensional space},
  Translations of Mathematical Monographs, vol.~93, American Mathematical
  Society, Providence, RI, 1991, Translated from the Russian by E. J. F.
  Primrose.

\bibitem[Gom87]{Gom87}
Jonas de~Miranda Gomes, \emph{Spherical surfaces with constant mean curvature
  in hyperbolic space}, Bol. Soc. Brasil. Mat. \textbf{18} (1987), no.~2,
  49--73.

\bibitem[Hsi82]{Hsi82}
Wu-yi Hsiang, \emph{On generalization of theorems of {A}. {D}. {A}lexandrov and
  {C}. {D}elaunay on hypersurfaces of constant mean curvature}, Duke Math. J.
  \textbf{49} (1982), no.~3, 485--496.

\bibitem[HW15]{Wang11a}
Zheng Huang and Biao Wang, \emph{Counting minimal surfaces in quasi-{F}uchsian
  three-manifolds}, Trans. Amer. Math. Soc. \textbf{367} (2015), no.~9,
  6063--6083.

\bibitem[L{\'o}p00]{Lop00}
Rafael L{\'o}pez, \emph{Hypersurfaces with constant mean curvature in
  hyperbolic space}, Hokkaido Math. J. \textbf{29} (2000), no.~2, 229--245.

\bibitem[LR85]{LR85}
Gilbert Levitt and Harold Rosenberg, \emph{Symmetry of constant mean curvature
  hypersurfaces in hyperbolic space}, Duke Math. J. \textbf{52} (1985), no.~1,
  53--59.

\bibitem[LR89]{LR89}
Francisco~J. L{\'o}pez and Antonio Ros, \emph{Complete minimal surfaces with
  index one and stable constant mean curvature surfaces}, Comment. Math. Helv.
  \textbf{64} (1989), no.~1, 34--43.

\bibitem[LR91]{LR91}
\bysame, \emph{On embedded complete minimal surfaces of genus zero}, J.
  Differential Geom. \textbf{33} (1991), no.~1, 293--300.

\bibitem[Mor81]{Mor81}
Hiroshi Mori, \emph{Minimal surfaces of revolution in {$H^{3}$} and their
  global stability}, Indiana Univ. Math. J. \textbf{30} (1981), no.~5,
  787--794.

\bibitem[Mor82]{Mor82}
\bysame, \emph{On surfaces of right helicoid type in {$H\sp{3}$}}, Bol. Soc.
  Brasil. Mat. \textbf{13} (1982), no.~2, 57--62.

\bibitem[MP11]{MP2011}
William~H. Meeks, III and Joaqu{\'{\i}}n P{\'e}rez, \emph{The classical theory
  of minimal surfaces}, Bull. Amer. Math. Soc. (N.S.) \textbf{48} (2011),
  no.~3, 325--407.

\bibitem[MP12]{MP2012}
\bysame, \emph{A survey on classical minimal surface theory}, University
  Lecture Series, vol.~60, American Mathematical Society, Providence, RI, 2012.

\bibitem[MR05]{MR2005}
William~H. Meeks, III and Harold Rosenberg, \emph{The uniqueness of the
  helicoid}, Ann. of Math. (2) \textbf{161} (2005), no.~2, 727--758.

\bibitem[MT98]{MT98}
Katsuhiko Matsuzaki and Masahiko Taniguchi, \emph{Hyperbolic manifolds and
  {K}leinian groups}, Oxford Mathematical Monographs, The Oxford University
  Press, New York, 1998.

\bibitem[MY82a]{MY1982(t)}
William~W. Meeks, III and Shing~Tung Yau, \emph{The classical {P}lateau problem
  and the topology of three-dimensional manifolds. {T}he embedding of the
  solution given by {D}ouglas-{M}orrey and an analytic proof of {D}ehn's
  lemma}, Topology \textbf{21} (1982), no.~4, 409--442.

\bibitem[MY82b]{MY1982(mz)}
\bysame, \emph{The existence of embedded minimal surfaces and the problem of
  uniqueness}, Math. Z. \textbf{179} (1982), no.~2, 151--168.

\bibitem[Sch83]{Sch83}
Richard~M. Schoen, \emph{Uniqueness, symmetry, and embeddedness of minimal
  surfaces}, J. Differential Geom. \textbf{18} (1983), no.~4, 791--809 (1984).

\bibitem[Seo11]{Seo11}
Keomkyo Seo, \emph{Stable minimal hypersurfaces in the hyperbolic space}, J.
  Korean Math. Soc. \textbf{48} (2011), no.~2, 253--266.

\bibitem[Tuz93]{Tuz93}
Alexey~A. Tuzhilin, \emph{Global properties of minimal surfaces in {${\mathbb
  R}\sp 3$} and {${\mathbb H}\sp 3$} and their {M}orse type indices}, Minimal
  surfaces, Adv. Soviet Math., vol.~15, Amer. Math. Soc., Providence, RI, 1993,
  pp.~193--233.

\bibitem[Wan12]{Wang10}
Biao Wang, \emph{Minimal surfaces in quasi-{F}uchsian 3-manifolds}, Math. Ann.
  \textbf{354} (2012), no.~3, 955--966.

\bibitem[Xin03]{Xin03}
Yuanlong Xin, \emph{Minimal submanifolds and related topics}, Nankai Tracts in
  Mathematics, vol.~8, World Scientific Publishing Co. Inc., River Edge, NJ,
  2003.

\end{thebibliography}
%--------------------------------------------------------------------
\end{document}